\documentclass[12pt]{amsart}
\usepackage{amssymb,amsmath}

\newcommand{\bC}{\mathbb{C}}

\newcommand{\bP}{\mathbb{P}}
\newcommand{\bQ}{\mathbb{Q}}
\newcommand{\bR}{\mathbb{R}}
\newcommand{\bZ}{\mathbb{Z}}

\newcommand{\cD}{{\mathcal{D}}}

\newcommand{\cL}{{\mathcal{L}}}

\newcommand{\cO}{{\mathcal{O}}}

\newcommand{\et}{\quad\mbox{and}\quad}

\newcommand{\GL}{\mathrm{GL}}

\newcommand{\lambdahat}{\hat{\lambda}}

\newcommand{\norm}[1]{\|\hspace*{2pt}#1\hspace*{1pt}\|}

\newcommand{\Qmult}{\bQ^{\times}}
\newcommand{\Qspan}[1]{\langle\,#1\rangle_\bQ}

\newcommand{\tE}{{\tilde{E}}}

\newcommand{\tq}{\tilde{q}}
\newcommand{\teta}{{\tilde{\eta}}}
\newcommand{\tueta}{{\tilde{\ueta}}}
\newcommand{\tux}{{\tilde{\ux}}}

\newcommand{\tuy}{{\tilde{\uy}}}
\newcommand{\txi}{{\tilde{\xi}}}

\newcommand{\tZ}{\tilde{Z}}
\newcommand{\tZli}{\tilde{Z}^{\mathrm{li}}}

\newcommand{\ue}{\mathbf{e}}

\newcommand{\uu}{\mathbf{u}}

\newcommand{\uv}{\mathbf{v}}
\newcommand{\uw}{\mathbf{w}}
\newcommand{\ux}{\mathbf{x}}
\newcommand{\uy}{\mathbf{y}}
\newcommand{\uz}{\mathbf{z}}
\newcommand{\ueta}{{\boldsymbol{\eta}}}
\newcommand{\uxi}{{\boldsymbol{\xi}}}

\newcommand{\disp}{\displaystyle}
\newcommand{\Zli}{Z^{\mathrm{li}}}

\DeclareMathOperator{\dist}{dist}

\newtheorem{theorem}{Theorem}[section]
\newtheorem{lemma}[theorem]{Lemma}
\newtheorem{proposition}[theorem]{Proposition}
\newtheorem{cor}[theorem]{Corollary}

\newtheorem{lemme}[theorem]{Lemme}

\theoremstyle{remark}

\numberwithin{equation}{section}

\begin{document}

\title[Rational approximation on quadratic hypersurfaces]
{Rational approximation to real points\\ on quadratic hypersurfaces}
\author{Anthony Po\"els}
\address{
   D\'epartement de Math\'ematiques\\
   Universit\'e d'Ottawa\\
   150 Louis Pasteur\\
   Ottawa, Ontario K1N 6N5, Canada}
\email{anthony.poels@uottawa.ca}
\author{Damien Roy}
\address{
   D\'epartement de Math\'ematiques\\
   Universit\'e d'Ottawa\\
   150 Louis Pasteur\\
   Ottawa, Ontario K1N 6N5, Canada}
\email{droy@uottawa.ca}
\subjclass[2010]{Primary 11J13; Secondary 11J82}
\thanks{Work of both authors partially supported by NSERC}

\keywords{Diophantine approximation, exponents of Diophantine
approximation, extremal points, Marnat-Moshchevitin transference principle,
quadratic hypersurfaces, rational points, simultaneous approximation.}

\begin{abstract}
Let $Z$ be a quadratic hypersurface of $\bP^n(\bR)$ defined over $\bQ$ containing
points whose coordinates are linearly independent over $\bQ$.  We show that,
among these points, the largest exponent of uniform rational approximation
is the inverse $1/\rho$ of an explicit Pisot number $\rho<2$ depending only
on $n$ if the Witt index (over $\bQ$) of the quadratic form $q$ defining $Z$ is at most $1$,
and that it is equal to $1$ otherwise. Furthermore there are points of $Z$
which realize this maximum. They constitute a countably infinite set in the
first case, and an uncountable set in the second case.  The proof for the upper
bound $1/\rho$ uses a recent transference inequality of Marnat and Moshchevitin.
In the case $n=3$, we recover results of the second author while for $n>3$,
this completes recent work of Kleinbock and Moshchevitin.
\end{abstract}

\maketitle

\baselineskip=16.7pt 

\section{Introduction}

Nowadays, we have a good knowledge on how well points in projective
$n$-space $\bP^n(\bR)$ can be approximated by rational points
i.e.~by points of $\bP^n(\bQ)$.  Thanks to recent advances in
parametric geometry of numbers and in metrical theory
\cite{DFSU,R2015,SS2013a}, we essentially know all possible ways
in which a given point in $\bP^n(\bR)$ behaves with respect to
rational approximation and, in good cases, we also know the Hausdorff
dimension of the set of exceptional points having a given pattern
of approximation.  However, the situation changes drastically if
we restrict to points from a proper algebraic subset $Z$
of $\bP^n(\bR)$ defined over $\bQ$.  In particular, for the type
of problem that we have in mind, very little is known about
algebraic curves of $\bP^n(\bR)$ defined over $\bQ$ of degree
$d\ge 2$ besides the case $d=2$ treated in \cite{Rconic} which
reduces to studying conics in $\bP^2(\bR)$. In this paper,
we extend the results of \cite{Rconic} to quadratic hypersurfaces
of $\bP^n(\bR)$ defined over $\bQ$, thus completing recent work
of Kleinbock and Moshchevitin \cite{KM2019}.  We adopt here the
projective setting as in \cite{L2009, Rconic} because it is
more conceptual and brings simplifications with respect to
the traditional but equivalent affine point of view.
The connection is explained in Section \ref{sec:NPR}.

To each point $\xi$ of $\bP^n(\bR)$, one attaches two numbers
which measure how well it is approximated by rational points.
Following the convention of Bugeaud and Laurent in \cite{BL2005},
they are the exponent of \emph{uniform} rational approximation
$\lambdahat(\xi)$ and the exponent of \emph{best} rational
approximation $\lambda(\xi)$.  We recall their precise definition
in the next section.  In this study, we may assume that each
representative of $\xi$ in $\bR^{n+1}$ has linearly independent
coordinates over $\bQ$.  Then these exponents satisfy
\[
 \frac{1}{n}\le \lambdahat(\xi)\le 1
 \et
 \lambdahat(\xi)\le \lambda(\xi)\le \infty.
\]
There is also a third inequality relating $\lambdahat(\xi)$
and $\lambda(\xi)$.  It was conjectured by Schmidt and Summerer
at the end of Section 3 in \cite{SS2013b}, and was recently
proved by Marnat and Moshchevitin in \cite{MM2019}.  It plays
a crucial role in \cite{KM2019} and here also but through a sharper
form, in terms of measures of approximation, from joint work
with Van Nguyen \cite{NPR}.
These results are recalled in Section \ref{sec:NPR}.

For each algebraic
subset $Z$ of $\bP^n(\bR)$, we denote by $\Zli$ the set of points
of $Z$ whose representatives in $\bR^{n+1}$ have $\bQ$-linearly
independent coordinates and, provided that this set is not empty,
we are interested in the following important invariant
\[
 \lambdahat(Z):=\sup\{\lambdahat(\xi)\,;\,\xi\in \Zli\} \in [1/n,1].
\]

A quadratic form on $\bQ^{n+1}$ is a map $q\colon \bQ^{n+1}\to\bQ$
given by a non-zero homogeneous polynomial of $\bQ[t_0,\dots,t_n]$
of degree $2$.  Its Witt index is the largest integer $m\ge 0$ such
that $\bQ^{n+1}$ contains an orthogonal sum of $m$ hyperbolic planes
for $q$ (see Section \ref{sec:results}).  For each quadratic form $q$
on $\bQ^{n+1}$, we denote by $q_\bR\colon\bR^{n+1}\to\bR$ its
extension to $\bR^{n+1}$ given by the same polynomial, and by $Z(q_\bR)$
the set of zeros of $q_\bR$ in $\bP^n(\bR)$.  By a quadratic
hypersurface of $\bP^n(\bR)$ defined over $\bQ$ we mean any non-empty
subset of $\bP^n(\bR)$ of this form.  In terms of exponents of
approximation, our main result reads as follows.

\begin{theorem}
\label{intro:thm}
Let $n\ge 2$ be an integer, let $Z$ be a quadratic
hypersurface of $\bP^n(\bR)$ defined over $\bQ$, and
let $m$ be the Witt index of the quadratic form
on $\bQ^{n+1}$ defining $Z$. If $\Zli$ is not empty, then
\[
 \lambdahat(Z)
 =\begin{cases}
   1/\rho_n &\text{if $m\le 1$,}\\
   1 &\text{else}
  \end{cases}
\]
where $\rho_n\in(1,2)$ denotes the unique positive root
of the polynomial $x^n-(x^{n-1}+\cdots+x+1)$.
Moreover, the set $\{\xi\in\Zli\,;\,\lambdahat(\xi)=\lambdahat(Z)\}$
is countably infinite if $m\le 1$ and uncountable otherwise.
\end{theorem}

In \cite[Theorem 1a]{KM2019} (resp.\ \cite[Theorem 2a]{KM2019}),
Kleinbock and Moshchevitin prove the upper bound
$\lambdahat(Z)\le 1/\rho_n$ for all quadratic hypersurfaces $Z$
of $\bP^n(\bR)$ defined by quadratic forms $q$ on $\bQ^{n+1}$
of the type
\begin{equation}
 \label{intro:eq:KM}
 q(t_0,t_1,\dots,t_n)=t_0^2-f(t_1,\dots,t_n)
 \quad
 \big(\text{resp.}\ q(t_0,t_1,\dots,t_n)=t_0t_n-f(t_1,\dots,t_{n-1})\big)
\end{equation}
where $f$ is a quadratic form on $\bQ^n$ (resp.\ $\bQ^{n-1}$) with no
nontrivial zero.  As we will see in Section \ref{sec:equiv},
the general case of a non-degenerate quadratic form $q$ of Witt
index $m\le 1$ can be reduced to these two special cases. Thus their
results yield $\lambdahat(Z)\le 1/\rho_n$ in the non-degenerate case
when $m\le 1$. Theorem \ref{intro:thm} shows that this extends with
an equality to all quadratic forms of Witt index $m\le 1$.  For example,
it applies to the degenerate quadratic forms $q(t_0,\dots,t_N)$
given by the same formulas \eqref{intro:eq:KM}
for each integer $N>n$.  In particular, for each $n\ge 2$, it yields
$\lambdahat(Z)=1/\rho_n$ for the hypersurface $Z$ of $\bP^n(\bR)$
of equation $t_0^2-2t_1^2=0$ (resp.\ $t_0t_2-t_1^2=0$).

When $n=2$, the number $\rho_n=\rho_2$ is the golden ratio and
we automatically have $m\le 1$.  Then, the above theorem reduces
to \cite[Theorem 1.2]{Rconic}.  In general, $\rho_n$ is a
Pisot number for each $n\ge 2$ (see Section \ref{sec:construction}).

In the next section we present a sharper version of the theorem
dealing with measures of approximation to points of $Z$ instead
of the coarse estimation provided by exponents of approximation.
In the case where $m\le 1$, it leads to a notion of extremal
points on $Z$ generalizing the notion of extremal numbers
from \cite{Rnote, RcubicI}.

%
%

\section{Main result and notation}
 \label{sec:results}

Fix an integer $n\ge 1$.  In this paper we endow $\bR^{n+1}$ with
the standard structure of Euclidean space for which the canonical
basis $(\ue_0,\dots,\ue_n)$ is orthonormal. More generally, for
each $k=1,\dots,n+1$, we endow its $k$-th exterior power
$\bigwedge^k\bR^{n+1}$ with the Euclidean space structure for
which the products $\ue_{i_1}\wedge\cdots\wedge\ue_{i_k}$
with $0\le i_1<\cdots<i_k\le n$ form an orthonormal basis.
In all cases we use the same symbol $\|\ \|$ to denote the
associated norm.

We denote by $(x_0:x_1:\cdots:x_n)$ or simply by $[\ux]$ the class
in $\bP^n(\bR)$ of a non-zero point $\ux=(x_0,x_1,\dots,x_n)$
of $\bR^{n+1}$.  Given non-zero points $\ux,\uy\in\bR^{n+1}$, we
define
\[
 \dist([\ux],[\uy])
  :=\dist(\ux,\uy)
  :=\frac{\norm{\ux\wedge\uy}}{\norm{\ux}\,\norm{\uy}}
\]
and call this ratio the \emph{projective distance} between $[\ux]$
and $[\uy]$ as it depends only on the classes of the points $\ux$
and $\uy$ in $\bP^n(\bR)$.  Geometrically, this is the sinus of the
acute angle between the lines $\bR\ux$ and $\bR\uy$ spanned by $\ux$
and $\uy$ in $\bR^{n+1}$.  It is well-known that this yields a metric
on $\bP^n(\bR)$ as it satisfies the triangle inequality
\[
 \dist([\ux],[\uz])\le \dist([\ux],[\uy])+\dist([\uy],[\uz])
\]
for any $\ux,\uy,\uz\in\bR^{n+1}\setminus\{0\}$.

Let $\xi\in\bP^n(\bR)$ and let $\uxi\in\bR^{n+1}$ be a representative
of $\xi$ so that $\xi=[\uxi]$. For each non-zero $\ux\in\bZ^{n+1}$, we
set
\[
 D_\xi(\ux)
   := \frac{\norm{\ux\wedge\uxi}}{\norm{\uxi}}
   = \norm{\ux}\dist(\xi,[\ux])
\]
and for each $X\ge 1$ we define
\[
 \cD_\xi(X)
  :=\min\left\{ D_\xi(\ux)\,;\,
      \ux\in\bZ^{n+1}\setminus\{0\}
      \ \text{and}\
      \norm{\ux}\le X \right\}.
\]
This minimum is achieved by a \emph{primitive} point
of $\bZ^{n+1}$, that is a point whose coordinates are relatively prime
as a set.  Thus, upon defining the \emph{height} of a point
in $\bP^n(\bQ)$ as the norm $\norm{\ux}$ of its primitive
representatives $\pm \ux$ in $\bZ^{n+1}$, we view $\cD_\xi(X)$
as a measure of approximation to $\xi$ by rational points of height
at most $X$.  We define $\lambdahat(\xi)$ (resp.\ $\lambda(\xi)$) to
be the supremum of all $\lambda\in\bR$ such that
$\cD_\xi(X)\le X^{-\lambda}$ for each sufficiently large $X$ (resp.\
for arbitrarily large values of $X$).  Equivalently, $\lambdahat(\xi)$
(resp.\ $\lambda(\xi)$) is the
supremum of all $\lambda\in\bR$ such that
\[
 \limsup_{X\to\infty} X^\lambda \cD_\xi(X)<\infty
 \qquad
 \big(\text{resp.}\ \liminf_{X\to\infty} X^\lambda \cD_\xi(X)<\infty\big).
\]
With this notation, we can add the following precision to
Theorem \ref{intro:thm}.

\begin{theorem}
 \label{results:thm1}
Under the hypotheses of Theorem \ref{intro:thm}, suppose that $m\le 1$.
Then we have
\begin{itemize}
 \item[(i)] $\limsup_{X\to\infty} X^{1/\rho_n}\cD_\xi(X) >0$
     for each $\xi\in \Zli$;
 \smallskip
 \item[(ii)] $\limsup_{X\to\infty} X^{1/\rho_n}\cD_\xi(X) <\infty$
     for infinitely many $\xi\in \Zli$.
\end{itemize}
\end{theorem}

Indeed, Part (i) implies that $\lambdahat(\xi)\le 1/\rho_n$ for
each $\xi\in\Zli$ while Part (ii) yields points $\xi\in\Zli$
with $\lambdahat(\xi)\ge 1/\rho_n$ and thus $\lambdahat(\xi)=1/\rho_n$.
Altogether, this means that $\lambdahat(Z)=1/\rho_n$.

We say that the points $\xi\in\Zli$ which satisfy the condition
in Part (ii) are the \emph{extremal points} of $Z$.  For any such
point, there exist constants $c_1\ge c_2>0$ such that
$\cD_\xi(X)\le c_1X^{-1/\rho_n}$ for each sufficiently large $X$,
as well as $\cD_\xi(X)\ge c_2X^{-1/\rho_n}$ for arbitrarily large
values of $X$.  This generalizes the notion of extremal numbers
from \cite{RcubicI} because $\rho_2$ is the golden ratio and thus
the extremal points of the conic $Z$ of $\bP^2(\bR)$ defined by
the quadratic form $t_0t_2-t_1^2$ of Witt index $m=1$ are the
points $(1:\xi:\xi^2)$ where $\xi$ is an extremal number.

For any subset $E$ of $\bP^n(\bQ)$ and any $X\ge 1$, we define
\[
 \cD_\xi(X;E)
  :=\min\left\{ D_\xi(\ux)\,;\,
      \ux\in \bZ^{n+1}\setminus\{0\},\
      \norm{\ux}\le X
      \ \text{and}\
      [\ux]\in E \right\}
\]
with the convention that $\min\emptyset=-\infty$. Our main result
below extends Theorem \ref{results:thm1} by taking $E$ to be
either the set $Z(\bQ):=Z\cap\bP^n(\bQ)$ of rational points of $Z$
or its complement.

\begin{theorem}
 \label{results:thm2}
Suppose that $n\ge 2$.  Let $Z$ be a quadratic hypersurface
of $\bP^n(\bR)$ defined over $\bQ$ with $\Zli\neq\emptyset$.
Define $Z(\bQ):=Z\cap\bP^n(\bQ)$ and $E:=\bP^n(\bQ)\setminus Z$. Then
\begin{itemize}
 \item[(i)] we have $\limsup_{X\to\infty} X^{1/\rho_n}\cD_{\xi}(X;E) >0$
   for each $\xi\in \Zli$;
 \smallskip
 \item[(ii)] we have $\limsup_{X\to\infty} X^{1/\rho_n}\cD_{\xi}(X;E) <\infty$
   for infinitely many $\xi\in \Zli$;
 \smallskip
 \item[(iii)] there exists $\epsilon>0$ such that the set
   \[
    \{\xi\in \Zli\,;\,\limsup_{X\to\infty}
       X^{1/\rho_n-\epsilon}\cD_{\xi}(X;E) <\infty\}
   \]
   is at most countable.
\end{itemize}
Moreover, let $m$ denote the Witt index of the quadratic form
on $\bQ^{n+1}$ defining $Z$.
\begin{itemize}
 \item[(iv)] If $m\le 1$, then any $\xi\in\Zli$ with
    $\lambdahat(\xi)>1/2$ satisfies $\cD_{\xi}(X)=\cD_{\xi}(X;E)$
    for each sufficiently large $X$.
 \item[(v)] If $m>1$, then any $\xi\in\Zli$ with
   $\lambdahat(\xi)>1/\rho_n$ satisfies $\cD_{\xi}(X)=\cD_{\xi}(X;Z(\bQ))$
   for each sufficiently large $X$.  Moreover, for each monotonically
   decreasing function $\varphi\colon[1,\infty)\to(0,1]$ with
   $\lim_{X\to\infty}\varphi(X)=0$ and $\lim_{X\to\infty}X\varphi(X)=\infty$,
   there are uncountably many $\xi\in\Zli$ satisfying $\cD_{\xi}(X;Z(\bQ))
   \le \varphi(X)$ for all sufficiently large $X$.
\end{itemize}
\end{theorem}

In view of (iv), when $m\le 1$, Parts (i) and (ii) yield
Theorem \ref{results:thm1} (because $1<\rho_n<2$). Moreover
the extremal points of $Z$ are the points $\xi\in\Zli$ 
satisfying the inequality of Part (ii).  
Taking this condition as the general 
definition of an extremal point $\xi$ of $Z$ (without the above 
restriction $m\le 1$), it follows from (iii) that the extremal 
points of $Z$ always form an infinite countable set.  Finally, 
if $m>1$, then applying (v) with $\varphi=\log(3X)/X$
yields uncountably many points $\xi\in\Zli$ with
$\lambdahat(\xi)\ge 1$ and so $\lambdahat(\xi)=1$.
Thus Theorem \ref{results:thm2} also implies Theorem \ref{intro:thm}.

The proof of Parts (i) and (iii) is given in
Section \ref{sec:outside} based on a result from joint work with
Van Nguyen \cite{NPR} that we recall in Section \ref{sec:NPR}.
This result complements the inequality of Marnat and Moshchevitin
from \cite{MM2019}, also recalled in Section \ref{sec:NPR}, which
relates $\lambda(\xi)$ and $\lambdahat(\xi)$ for any
$\xi\in(\bP^n(\bR))^{\mathrm{li}}$.  The proof of Parts (ii), (iv)
and (v) are given respectively in Sections \ref{sec:construction},
\ref{sec:proofiv} and \ref{sec:proofv} respectively.  The preliminary
sections \ref{sec:forms} and \ref{sec:equiv} present general facts
about quadratic forms.  In Section \ref{sec:forms}, we recall the
Witt decomposition of a quadratic space $V$ over $\bQ$ and define
a morphism $\psi\colon V\times V\to V$ with crucial properties.
Section \ref{sec:equiv} reduces the study of a general quadratic
form to two types of forms, a fact that we use in Section
\ref{sec:construction} for the construction of extremal points.
Finally, Section \ref{sec:est} provides various estimates needed
throughout the paper.

%
%

\section{Preliminaries on quadratic forms}
\label{sec:forms}

Let $V$ be a vector space over $\bQ$ of finite dimension
$n+1$ for some integer $n\ge 0$.  A \emph{quadratic form}
on $V$ is a map $q\colon V\to\bQ$ given by
\begin{equation}
 \label{forms:eq:qb}
 q(\ux)=\frac{1}{2}\, b(\ux,\ux) \quad\text{for each $\ux\in V$,}
\end{equation}
for some non-zero symmetric bilinear form
$b\colon V\times V\to \bQ$.  This bilinear form is in turn
uniquely determined by $q$ through the formula
\begin{equation}
 \label{forms:eq:bq}
 b(\ux,\uy)=q(\ux+\uy)-q(\ux)-q(\uy)
 \quad\text{for each $\ux,\uy\in V$.}
\end{equation}
When $V=\bQ^{n+1}$, which is the main case of interest for us,
this is equivalent to the definition given in the introduction.
Moreover, the normalization factor $1/2$ in \eqref{forms:eq:qb}
ensures that, if $q$ is integer-valued on $\bZ^{n+1}$, then
$b$ also is integer-valued on $\bZ^{n+1}\times\bZ^{n+1}$
(by \eqref{forms:eq:bq}).

A useful consequence of the formula \eqref{forms:eq:qb} is that
\begin{equation}
 \label{forms:eq:qcomb}
 q(s\ux+t\uy)=s^2q(\ux)+stb(\ux,\uy)+t^2q(\uy)
\end{equation}
for any $s,t\in\bQ$ and any $\ux,\uy\in V$. For the choice
of $s=b(\ux,\uy)$ and $t=-q(\ux)$, it yields the following
generalization of \cite[Lemma 4.1]{Rconic}.

\begin{lemma}
\label{forms:lemma:psi}
For each choice of $\ux,\uy\in V$, the point
\[
 \uz = \psi(\ux,\uy) := b(\ux,\uy)\ux-q(\ux)\uy \in V
\]
satisfies $q(\uz)=q(\ux)^2q(\uy)$ and $\psi(\ux,\uz)=q(\ux)^2\uy$.
\end{lemma}

In particular, we have $q(\uz)=0$ if $q(\uy)=0$ and $q(\uz)=1$
if $q(\ux)=q(\uy)=1$. The polynomial map
$\psi\colon V\times V\to V$ so defined is central to the
present work.  Note that it is bi-homogeneous of degree $(2,1)$.

Given $q$ and $b$ as above, we say that points $\ux$, $\uy$
of $V$ are \emph{orthogonal} (with respect to $q$)
if $b(\ux,\uy)=0$.  It is well-known that every subspace $W$
of $V$ admits an \emph{orthogonal basis}, namely a basis
whose elements are pairwise orthogonal.  In particular,
$V$ itself admits an orthogonal basis $(\ux_0,\dots,\ux_n)$
and upon setting $a_i=q(\ux_i)$ for $i=0,\dots,n$, we obtain
\begin{equation}
 \label{forms:eq:diag}
 q(t_0\ux_0+\cdots+t_n\ux_n)=a_0t_0^2+\cdots+a_nt_n^2
\end{equation}
for each $(t_0,\dots,t_n)\in\bQ^{n+1}$.  A theorem of Silvester
(valid more generally over $\bR$) tells us that in such
a formula the number $n_0$ (resp.~$n_+$, resp.~$n_-$)
of indices $i$ with $a_i=0$ (resp.~$a_i>0$, resp.~$a_i<0$) is
independent of the basis. The difference $n-n_0$ is called the
\emph{rank} of $q$.

The \emph{orthogonal} of a subspace $U$ of $V$ is the
subspace $U^\perp$ given by
\[
 U^\perp:=\{\ux\in V\,;\, b(\ux,\uy)=0 \text{ for each }\uy\in U\}.
\]
In particular the subspace $\ker(q):=V^\perp$ is called
the \emph{kernel} of $q$.  It has dimension $n_0$.  We say
that $q$ is \emph{non-degenerate} if $\ker(q)=\{0\}$ and
\emph{degenerate} otherwise.  In general, we say that a
subspace $U$ of $V$ is \emph{non-degenerate} if
$U\cap U^\perp=\{0\}$, that it is \emph{totally isotropic}
if $U\subseteq U^\perp$, and that it is \emph{anisotropic}
if it contains no zero of $q$ else than $0$.
A \emph{hyperbolic plane} of $V$ is any non-degenerate and
non-anisotropic subspace of $V$ of dimension $2$. Equivalently,
this is a subspace $H$ of $V$ which admits a basis $\{\ux,\uy\}$
with $q(\ux)=q(\uy)=0$ and $b(\ux,\uy)\neq 0$.   We say that
subspaces $U$ and $W$ of $V$ are \emph{orthogonal}
if $W\subseteq U^\perp$.  Finally we say that a subspace $W$
of $V$ is an \emph{orthogonal sum} of subspaces $W_1,\dots,W_s$
if these subspaces are pairwise orthogonal and if $W$ is their
direct sum.  We then express this property as
$W = W_1 \perp \cdots \perp W_s$.  A theorem
of Witt \cite[Chapter XIV, \S 5, Corollary 5]{La1984} tells
us that there is a unique integer $m\ge 0$, called
the \emph{Witt index} of $q$ such that
\[
 V = \ker(q) \perp H_1\perp \cdots \perp H_m \perp W
\]
where $H_1,\dots,H_m$ are hyperbolic planes of $V$ and $W$
is an anisotropic subspace of $V$.  Another characterization
of $m$ is that all maximal totally isotropic subspaces of $V$
have dimension $m+\dim_\bQ\ker(q)$.

%
%

\section{Equivalent forms}
\label{sec:equiv}

Two quadratic forms $q$ and $\tq$ on $\bQ^{n+1}$ are said to be
\emph{equivalent} if $\tq=q\circ T$ for some invertible linear
operator $T$ on $\bQ^{n+1}$.  Then $q$ and $\tq$ have the same rank
and the same Witt index.  Moreover, the extended quadratic forms
$q_\bR$ and $\tq_\bR$ on $\bR^{n+1}$ satisfy
$\tq_\bR = q_\bR\circ T_\bR$ where $T_\bR$ denotes the invertible
linear operator on $\bR^{n+1}$ extending $T$.  For the associated
zero sets $Z=Z(q_\bR)$ and $\tZ=Z(\tq_\bR)$ in $\bP^n(\bR)$, this
implies that
\begin{equation}
 \label{equiv:Z}
 Z=T_\bR(\tZ),
 \quad
 \Zli=T_\bR(\tZ^\mathrm{li})
 \et
 Z\cap\bP^n(\bQ)=T(\tZ\cap\bP^n(\bQ)),
\end{equation}
using the same symbols $T$ (resp.~$T_\bR$) to denote the automorphism
of $\bP^n(\bQ)$ (resp.~$\bP^n(\bR)$) induced by $T$ (resp.~$T_\bR$).

Standard arguments as in \cite[Section 2]{Rconic} also yield
the following estimates.

\begin{lemma}
\label{equiv:lemma}
Let $\xi\in\bP^n(\bR)$ and let $E$ be a non-empty subset
of $\bP^n(\bQ)$.  Put $\txi=T_\bR(\xi)$ and $\tE=T(E)$ for
some $T\in\GL_{n+1}(\bQ)$. Then, for each $X\ge 1$, we have
\[
 \cD_\txi(X;\tE) \asymp \cD_\xi(X;E)
\]
with implied constants depending only on $T$.
\end{lemma}

Combining this lemma with the preceding observations, we deduce
that, if Theorem \ref{results:thm2} holds for some quadratic
hypersurface $Z$ attached to a quadratic form $q$ on $\bQ^{n+1}$,
then it also holds for the quadratic hypersurface $\tZ$ attached
to any quadratic form $\tq$ that is equivalent to $q$ or more
generally to $aq$ for some $a\in\Qmult$.  This reduction will be
useful in the proof of Theorem \ref{results:thm2} (ii).

Keeping the same notation, we also deduce from \eqref{equiv:Z}
that $Z$ (resp.~$\Zli$) is not empty if and only if the same is
true of $\tZ$ (resp.~$\tZ^\mathrm{li}$).  This observation
yields the following criterion.

\begin{proposition}
\label{equiv:prop:Zli}
Let $q\colon\bQ^{n+1}\to\bQ$ be a (non-zero) quadratic form, and
let $Z=Z(q_\bR)$.  Then $\Zli\neq\emptyset$ if
and only if there exists $a\in\bQ^\times$ such that $aq$ is
equivalent to
\begin{equation}
 \label{equiv:eq:forme_can1}
 t_0^2-a_1t_1^2-\cdots-a_{n}t_{n}^2
\end{equation}
for some integers $a_1,\dots,a_{n}$ where $a_1>0$ is not a square,
or equivalent to
\begin{equation}
 \label{equiv:eq:forme_can2}
 t_0t_1-a_2t_2^2-\cdots-a_{n}t_{n}^2
\end{equation}
for some integers $a_2,\dots,a_{n}$ with $a_2\neq0$.
\end{proposition}

Note that the two cases are not mutually exclusive.  However,
in the second case, the Witt index of $q$ is at least $1$.

\begin{proof}
Suppose first that $\Zli\neq\emptyset$.  Then $q$ is equivalent
to $a_0t_0^2+\cdots+a_nt_n^2$ for some rational numbers $a_0,\dots,a_n$
among which at least one is positive and at least one is negative.
By permuting the variables if necessary we may assume that $a_0>0$
and $a_1<0$.  If $-a_0a_1$ is a square, the polynomial $a_0t_0^2+a_1t_1^2$
is not irreducible over $\bQ$ and then we may also assume that $a_2\neq 0$.
Hence by a further diagonal change of variables we obtain that $a_0q$
is equivalent to $t_0^2-a'_1t_1^2-\cdots-a'_{n}t_{n}^2$ for some
integers $a'_1,\dots,a'_n$ satisfying $a'_1>0$ and also $a'_2\neq 0$
when $a'_1$ is a square.  In the latter case $t_0^2-a'_1t_1^2$ is
equivalent to $t_0t_1$ and so $a_0q$ is equivalent to
$t_0t_1-a'_2t_2^2-\cdots-a'_{n}t_{n}^2$ where $a_2'\neq 0$.

Conversely, assume that $aq$ is equivalent to the quadratic form $\tq$
given by \eqref{equiv:eq:forme_can1} or \eqref{equiv:eq:forme_can2} for
an appropriate choice of coefficients.  Then $\tq_\bR$ admits
a zero $\uxi=(\xi_0,\xi_1,\dots,\xi_n)$ in $\bR^{n+1}$ with
$\xi_1,\dots,\xi_n$ algebraically independent over $\bQ$. If the
coordinates of $\uxi$ are not linearly independent over $\bQ$, then
$\xi_0=c_1\xi_1+\cdots+c_n\xi_n$ for some $c_1,\dots,c_n\in\bQ$ and
so $(\xi_1,\dots,\xi_n)$ is a zero of the polynomial
$\tq(c_1t_1+\cdots+c_nt_n,t_1,\dots,t_n)$. However, this polynomial
is non-zero because, if $\tq$ is given
by \eqref{equiv:eq:forme_can1}, its coefficient of $t_1^2$
is $c_1^2-a_1\neq 0$ and, if $\tq$ is given
by \eqref{equiv:eq:forme_can2}, its coefficient of $t_2^2$
is $-a_2\neq 0$.  This contradiction means that $[\uxi]$ belongs to
$(Z(\tq_\bR))^\mathrm{li}$ and so $\Zli\neq\emptyset$.
\end{proof}

The next corollary explains the assertion made in the introduction,
to the effect that the results \cite[Theorem 1a and 2a]{KM2019}
of Kleinbock and Moshchevitin yield $\lambdahat(Z)\le 1/\rho_n$
for any quadratic hypersurface $Z$ of $\bP^n(\bR)$ with
$\Zli\neq\emptyset$, defined by a non-degenerate quadratic form
on $\bQ^{n+1}$ of Witt index $m\le 1$.  Note, on the way, that if
$m=0$ then we necessarily have $n\le 3$ by a theorem of Meyer
\cite[Chapter IV, Section 3.2, Corollary 2]{Se1973}.

\begin{cor}
\label{equiv:cor:KM}
Let $q\colon\bQ^{n+1}\to\bQ$ be a non-degenerate quadratic form, let
$m$ be its Witt index, and let $Z=Z(q_\bR)$.  Then we have both
$m\le 1$ and $\Zli\neq\emptyset$ if and only if there exists
$a\in\bQ^\times$ such that $aq$ is equivalent to
$t_0^2-f(t_1,\dots,t_{n})$ for some quadratic form $f\colon\bQ^n\to\bQ$
with no non-trivial zero and $f(1,0,\dots,0)>0$, or equivalent to
$t_0t_1-g(t_2,\dots,t_{n})$ for some quadratic form $g\colon\bQ^{n-1}\to\bQ$
with no non-trivial zero.
\end{cor}

\begin{proof}
This follows from the fact that a quadratic form on $\bQ^{n+1}$ of the
form \eqref{equiv:eq:forme_can1} (resp.~\eqref{equiv:eq:forme_can2})
is non-degenerate of Witt index $0$ or $1$ if and only if the quadratic
form $a_1t_1^2+\cdots+a_nt_n^2$ (resp.~$a_2t_2^2+\cdots+a_nt_n^2$) has
no non-trivial zero on $\bQ^n$ (resp.~$\bQ^{n-1}$).
\end{proof}

We conclude with the following result which derives from similar
considerations.

\begin{proposition}
\label{equiv:prop:evitement}
Let $q\colon V\to\bQ$ be a quadratic form of Witt index at least $1$
and rank at least $3$ on some finite dimensional vector space $V$
over $\bQ$.  Then, for any finite set of proper subspaces of $V$,
there is a zero of $q$ in $V$ which lies outside of their union.
\end{proposition}

\begin{proof}
We have $\dim_\bQ(V)=n+1$ for some integer $n\ge 2$. By composing
$q$ with a suitable linear isomorphism from $\bQ^{n+1}$ to $V$, we
may assume that $V=\bQ^{n+1}$ and that
$q(t_0,\dots,t_n)=t_0t_1-a_2t_2^2-\cdots-a_nt_n^2$
for some $a_2,\dots,a_n\in\bQ$ with $a_2\neq 0$.
Then the polynomial map $\varphi\colon\bQ^{n-1}\to\bQ^{n+1}$ given by
\[
 \varphi(t_2,\dots,t_n)=(1,a_2t_2^2+\cdots+a_nt_n^2,t_2,\dots,t_n)
\]
has image in the zero set of $q$ in $\bQ^{n+1}$.  Moreover its components
are linearly independent over $\bQ$ as elements of $\bQ[t_2,\dots,t_n]$.
Thus the composite $\ell\circ\varphi\colon\bQ^{n+1}\to\bQ$ is a non-zero
polynomial function of degree at most $2$ for each non-zero linear
form $\ell$ on $\bQ^{n+1}$.  Consequently $\prod_{i=1}^s\ell_i\circ\varphi$
is a non-zero polynomial function of degree at most $2s$ on $\bQ^{n+1}$
for any non-zero linear forms $\ell_1,\dots,\ell_s$.  This means that
there is a point in the image of $\varphi$, thus a zero of $q$, which
avoids $\bigcup_{i=1}^s\ker(\ell_i)$.  The conclusion follows because
any proper subspace of $\bQ^{n+1}$ is contained in the kernel
of a non-zero linear form on $\bQ^{n+1}$.
\end{proof}

Note that the above result also follows from the general avoidance lemma
\cite[Prop.~7.5]{GR2017} of Gaudron and R\'emond, upon showing first
that the set of zeros of $q$ in $V$ is not itself a union of
hyperplanes of $V$.

%
%

\section{A complement to the inequality of Marnat and Moschevitin}
\label{sec:NPR}

We first explain how the definitions of Section \ref{sec:results}
relate to those used in \cite{NPR}.  To this end,
we fix a point $\uxi=(\xi_0,\dots,\xi_n)\in\bR^{n+1}$ with
$\xi_0\neq 0$, and we set $\xi = [\uxi]$.  Then, for any
$\ux=(x_0,\dots,x_n)\in\bZ^{n+1}$, we have
\begin{equation}
 \label{NPR:def:Lxi}
 D_\xi(\ux)
   \asymp \max_{0\le j< i\le n}|x_j\xi_i-x_i\xi_j|
   \asymp L_\uxi(\ux) := \max_{1\le i\le n}|x_0\xi_i-x_i\xi_0|
\end{equation}
where the implied constants depend only on $\uxi$ (and $n$).
This yields
\[
 \cD_\xi(X)
   \asymp
   \cL_\uxi(X) := \min\Big\{L_\uxi(\ux)\,;\,
             \ux\in\bZ^{n+1}\setminus\{0\}\ \text{and}\ \|\ux\|\le X\Big\}
 \quad (X\ge 1).
\]
Thus we may replace $\cD_\xi$ by $\cL_\uxi$ in the definition of both
$\lambda(\xi)$ and $\lambdahat(\xi)$ as well as in the statement of
Theorem \ref{results:thm1}.  In particular, $\lambda(\xi)$
(resp.~$\lambdahat(\xi)$) is the supremum of all $\lambda\ge 0$ such that
\[
 \liminf_{X\to\infty}X^\lambda\cL_\uxi(X) < \infty
 \quad
(\text{resp.}~\limsup_{X\to\infty}X^\lambda\cL_\uxi(X) < \infty).
\]
More generally, for any subset $S$ of $\bZ^{n+1}$ and any $X\ge 1$,
we define
\[
 \cL_\uxi(X;S)
 := \min\Big\{L_\uxi(\ux)\,;\,
             \ux\in S\setminus\{0\}
             \ \text{and}\
             \|\ux\|\le X\Big\}
\]
(with the convention that $\min\emptyset=-\infty$).  Then, for any
non-empty subset $E$ of $\bP^n(\bQ)$, we have
\[
 \cD_\xi(X;E)\asymp \cL_\uxi(X;S),
\]
where $S=\{\ux\in\bZ^{n+1}\setminus\{0\}\,;\,[\ux]\in E\}$.

We now quote the following general result from joint work with
Van~Nguyen \cite{NPR}; it will be our main tool in proving
Parts (i) and (iii) of Theorem \ref{results:thm2}.

\begin{theorem}
 \label{NPR:thm}
Let $\uxi$ be a point of\/ $\bR^{n+1}$ whose coordinates are
linearly independent over $\bQ$ and let $S\subseteq\bZ^{n+1}$.
Suppose that $n\ge 2$ and that there exist positive real numbers
$a$, $b$, $\alpha$, $\beta$ such that
\begin{equation}
 \label{NPR:thm:eq1}
 bX^{-\beta} \le \cL_\uxi(X;S) \le aX^{-\alpha}
\end{equation}
for each sufficiently large real number $X$.  Then we
have $\alpha\le \beta$ and
\begin{equation}
 \label{NPR:thm:eq2}
 \epsilon:= 1- \Big(\alpha + \frac{\alpha^2}{\beta}
                +\cdots+\frac{\alpha^n}{\beta^{n-1}}\Big) \ge 0.
\end{equation}
Moreover, there exists a constant $C>0$ which depends only
on $\uxi,a,b,\alpha,\beta$ with the following property. If
\begin{equation}
 \label{NPR:thm:eq3}
 \epsilon
  \le \frac{1}{4n}\Big(\frac{\alpha}{\beta}\Big)^n
      \min\{\alpha,\beta-\alpha\},
\end{equation}
then there is an unbounded sequence $(\uy_i)_{i\geq 0}$ of
non-zero integer points in $S$ which for each $i\ge 0$ satisfies
the following conditions:
\begin{itemize}
 \item[(i)]
  $\big|\alpha\log\norm{\uy_{i+1}} - \beta\log\norm{\uy_i}\big|
   \le C + 4\epsilon(\beta/\alpha)^n\log\norm{\uy_{i+1}}$;
 \smallskip
 \item[(ii)] $\big|\log L_\uxi(\uy_i) + \beta\log \norm{\uy_{i}}\big|
   \le C+4\epsilon(\beta/\alpha)^2\log \norm{\uy_i}$;
 \smallskip
 \item[(iii)] $\det(\uy_i,\dots,\uy_{i+n})\neq 0$;
 \smallskip
 \item[(iv)] there exists no $\ux\in S\cap\bZ^{n+1}\setminus\{0\}$
   with $\norm{\ux} < \norm{\uy_i}$ and $L_\uxi(\ux)\le L_\uxi(\uy_i)$.
\end{itemize}
\end{theorem}

Let $\xi=[\uxi]$.  Then, as explained
in \cite{NPR}, the first assertion of the theorem yields
the inequality
\[
 \lambdahat(\xi) + \frac{\lambdahat(\xi)^2}{\lambda(\xi)}+\cdots
      + \frac{\lambdahat(\xi)^{n}}{\lambda(\xi)^{n-1}}
     \leq 1
\]
due to Marnat and Moschevitin \cite[Theorem 1]{MM2019}, where the ratio
$\lambdahat(\xi)/\lambda(\xi)$ is interpreted as $0$ when $\lambda(\xi)=\infty$.
Note that these authors work in an affine setting which amounts to
write $\xi=(1:\xi_1:\cdots:\xi_n)$ with $\xi_0=1$.

\begin{cor}
 \label{NPR:cor}
With the same notation, suppose further that $\epsilon=0$.  Then we have
\[
 \limsup_{X\to\infty} X^\alpha\cL_\uxi(X;S) > 0.
\]
\end{cor}

This follows from \cite[Theorem 1.2]{NPR} but it is instructive to derive
it directly from the above theorem since similar arguments will be
needed later.  The proof given below uses standard estimates for a
determinant as in \cite[Lemma 9]{DS1969} (see Section \ref{sec:est}).

\begin{proof}
If $\alpha=\beta$, the conclusion is immediate since \eqref{NPR:thm:eq1}
then yields $X^\alpha \cL_\uxi(X;S)\ge b$ for each large
enough $X$.  We may therefore assume that $\alpha<\beta$.
Since $\epsilon=0$, the sequence $(\uy_i)_{i\ge 0}$ provided by
Theorem \ref{NPR:thm} satisfies
\[
 \norm{\uy_{i+1}} \asymp \norm{\uy_i}^{\beta/\alpha}
 \et
 L_\uxi(\uy_i) \asymp \norm{\uy_i}^{-\beta}
\]
with implied constants that are independent of $i$.
Put $X=\norm{\uy_{i+n}}/2$ for some arbitrarily large index $i$
and choose a non-zero point $\ux$ in $S$ with
$\norm{\ux}\le X$, such that $L_\uxi(\ux)=\cL_\uxi(X;S)$.
Assuming $i$ large enough, the points $\uy_i,\dots,\uy_{i+n}$ are
linearly independent over $\bQ$ (by (iii)) with
\[
 \norm{\uy_i}<\norm{\uy_{i+1}}<\dots<\norm{\uy_{i+n}}
 \et
 L_\uxi(\uy_i)>L_\uxi(\uy_{i+1})>\cdots>L_\uxi(\uy_{i+n})
\]
(since $\beta/\alpha>1$).  Moreover, since $\norm{\ux}
<\norm{\uy_{i+n}}$, we have $L_\uxi(\ux) > L_\uxi(\uy_{i+n})$
by the minimality condition (iv).  In particular, the points
$\uy_{i+n}$ and $\ux$ are linearly
independent and so there exists an index $j$ with $i\le j<i+n$
such that $\uy_i,\dots,\widehat{\uy_j},\dots,\uy_{i+n},\ux$ form a
basis of $\bQ^{n+1}$, where the hat on $\uy_j$ means that this point
is omitted from the list.  As $\uy_{i+n}$ realizes the maximum norm
and the smallest value for $L_\uxi$ among these integer points,
we deduce that
\begin{align*}
 1  \le |\det(\uy_i,\dots,\widehat{\uy_j},\dots,\uy_{i+n},\ux)|
   &\ll \norm{\uy_{i+n}}
        L_\uxi(\uy_i)\cdots\widehat{L_\uxi(\uy_j)}
           \cdots L_\uxi(\uy_{i+n-1})L_\uxi(\ux)\\
   &\ll \norm{\uy_{i+n}}
        L_\uxi(\uy_i)\cdots L_\uxi(\uy_{i+n-2})L_\uxi(\ux)\\
   &\ll X^{1-\beta(\alpha^n/\beta^n+\cdots+\alpha^2/\beta^2)}\cL_\uxi(X;S)
   \,=\, X^\alpha\cL_\uxi(X;S).
\end{align*}
The conclusion follows by letting $i$ tend to infinity.
\end{proof}

%
%

\section{Metrical estimates}
\label{sec:est}

Let $\uxi=(\xi_0,\dots,\xi_n)\in\bR^{n+1}$ with $\xi_0\neq 0$,
and let $q$ be a quadratic form on $\bQ^{n+1}$ with associated symmetric
bilinear form $b$.  In this section, we collect several estimates
for polynomial maps which follow from Taylor expansions about the
point $\uxi$.  We start with the following generalization of
\cite[Lemma 4.2]{Rconic}.

\begin{lemma}
 \label{est:lemma}
Suppose that $q(\uxi)=0$.  Then, for each $\ux,\uy\in\bZ^{n+1}$,
we have
\begin{itemize}
 \item[(i)] $|b(\ux,\uy)|\ll \|\uy\|L_\uxi(\ux)+\|\ux\|L_\uxi(\uy)$,
 \item[(ii)] $|q(\ux)|\ll \|\ux\| L_\uxi(\ux)$,
\end{itemize}
and the point $\uz:=\psi(\ux,\uy)=b(\ux,\uy)\ux-q(\ux)\uy$ satisfies
\begin{itemize}
 \item[(iii)] $L_\uxi(\uz)
     \ll \|\uy\|L_\uxi(\ux)^2 + \|\ux\|L_\uxi(\ux)L_\uxi(\uy)$,
 \item[(iv)] $\|\uz\|
     \ll \|\ux\|^2L_\uxi(\uy)+\|\uy\|L_\uxi(\ux)^2
         +\|\ux\|L_\uxi(\ux)L_\uxi(\uy)$,
\end{itemize}
all implied constants depending only on $q$ and $\uxi$.
\end{lemma}

\begin{proof}
We may assume that $\xi_0=1$.  Then, upon denoting by $x_0$
the first coordinate of $\ux$, we have $L_\uxi(\ux)\asymp \|\Delta\ux\|$
where $\Delta\ux=\ux-x_0\uxi$.  Using similar notation
for $\uy$ and $\uz$, we find
\begin{equation}
 \label{est:lemma:psi:eq1}
 b(\ux,\uy)
  = b(x_0\uxi+\Delta\ux,y_0\uxi+\Delta\uy)
  = y_0b(\Delta\ux,\uxi)+x_0b(\uxi,\Delta\uy)+b(\Delta\ux,\Delta\uy)
\end{equation}
which yields (i).  As a special case, we obtain
\begin{equation}
 \label{est:lemma:psi:eq2}
 q(\ux)=\frac{1}{2}b(\ux,\ux)
  = x_0b(\Delta\ux,\uxi)+q(\Delta\ux)
\end{equation}
which in turn yields (ii).  Since $\uz=b(\ux,\uy)\ux-q(\ux)\uy$, we get
\[
 L_\uxi(\uz)
  \asymp \|\Delta\uz\|
  = \|b(\ux,\uy)\Delta\ux - q(\ux)\Delta\uy\|
  \ll |b(\ux,\uy)|L_\uxi(\ux) + |q(\ux)|L_\uxi(\uy)
\]
which, together with the estimates (i) and (ii), leads to (iii).
Finally, using \eqref{est:lemma:psi:eq1} and
\eqref{est:lemma:psi:eq2}, we obtain
\[
 z_0 = b(\ux,\uy)x_0-q(\ux)y_0
     = x_0^2b(\uxi,\Delta\uy)+x_0b(\Delta\ux,\Delta\uy)
       -y_0q(\Delta\ux)
\]
which implies $|z_0|\ll \|\ux\|^2L_\uxi(\uy)+\|\uy\|L_\uxi(\ux)^2$
and (iv) follows since $\|\uz\|\le \|z_0\uxi\|+\|\Delta\uz\|$.
\end{proof}

Similarly, using the multilinearity of the wedge product, we obtain
\begin{equation}
\label{est:eq:wedge}
 \norm{\ux_1\wedge\dots\wedge\ux_k}
 \ll \sum_{i=1}^k
      \norm{\ux_i}
      L_\uxi(\ux_1)\cdots\widehat{L_\uxi(\ux_i)}\cdots L_\uxi(\ux_k)
\end{equation}
for any $k=1,\dots,n+1$ and any $\ux_1,\dots,\ux_k\in\bZ^{n+1}$,
with implied constants depending only on $\uxi$ and $n$.
When $\norm{\ux_k} = \max_i \norm{\ux_i}$ and
$L_\uxi(\ux_k) = \min_i L_\uxi(\ux_i)$, this simplifies to
\[
 \norm{\ux_1\wedge\dots\wedge\ux_k}
 \ll \norm{\ux_k}L_\uxi(\ux_1)\cdots L_\uxi(\ux_{k-1}).
\]
For $k=n+1$, these represent upper bounds for
$|\det(\ux_1,\dots,\ux_{n+1})|$ as in \cite[Lemma 9]{DS1969}.

We will also need the inequality
\begin{equation}
\label{est:eq:contraction}
 \norm{(\uu\cdot\ux)\uy-(\uu\cdot\uy)\ux}
   \le \norm{\uu}\norm{\ux\wedge\uy}
\end{equation}
valid for any $\uu,\ux,\uy\in\bR^{n+1}$, where the dot
represents the usual scalar product in $\bR^{n+1}$.  This
is immediate when $\uu$ is the point
$\ue_1=(1,0,\dots,0)$, and the general case follows
by applying to all vectors a rotation
mapping $\uu$ to $\norm{\uu}\ue_1$.

We conclude with the following simple criterion of linear
independence.

\begin{lemma}
 \label{est:lemma:li}
Suppose that $\lim_{i\to\infty} D_\xi(\ux_i)=0$ for a point
$\xi$ in $\bP^n(\bR)$ and a sequence of non-zero integer
points $(\ux_i)_{i\ge 1}$ in $\bZ^{n+1}$.  Then, the representatives
$\uxi$ of $\xi$ in $\bR^{n+1}$ have linearly independent
coordinates over $\bQ$ if and only if the subsequence
$(\ux_i)_{i\ge i_0}$ spans $\bQ^{n+1}$ for each $i_0\ge 1$.
\end{lemma}

\begin{proof}
There is no loss of generality in choosing $\uxi$ with
$\norm{\uxi}=1$.  Then, for each $\uu\in\bZ^{n+1}$, the
inequality \eqref{est:eq:contraction} yields
\begin{equation}
 \label{est:lemme:li:eq}
 \norm{(\uu\cdot\ux_i)\uxi-(\uu\cdot\uxi)\ux_i}
 \le \norm{\uu} \norm{\ux_i\wedge\uxi}
 = \norm{\uu}D_\xi(\ux_i)
 \quad (i\ge 1).
\end{equation}
If $\uxi$ has linearly dependent coordinates, we may choose
$\uu\neq 0$ with $\uu\cdot\uxi=0$.  Then \eqref{est:lemme:li:eq}
simplifies to $|\uu\cdot\ux_i|\le \norm{\uu}D_\xi(\ux_i)$.  So
the integer $\uu\cdot\ux_i$ vanishes for all sufficiently
large $i$, say for each $i\ge i_0$, and therefore $(\ux_i)_{i\ge i_0}$
does not span $\bQ^{n+1}$.  Conversely, if the subsequence
$(\ux_i)_{i\ge i_0}$ does not span $\bQ^{n+1}$ for some $i_0\ge 1$,
we may choose $\uu\neq 0$ such that $\uu\cdot\ux_i=0$ for each $i\ge i_0$.
Then \eqref{est:lemme:li:eq} yields $|\uu\cdot\uxi|\le
|\uu\cdot\uxi|\,\norm{\ux_i}\le \norm{\uu}D_\xi(\ux_i)$ for
each $i\ge i_0$.  This implies that $\uu\cdot\uxi=0$, and so $\uxi$
has linearly dependent coordinates.
\end{proof}

%
%

\section{Approximation by rational points outside of $Z$}
\label{sec:outside}

The goal of this section is to prove Parts (i) and (iii)
of Theorem \ref{results:thm2}.  So, we assume that $n\ge 2$
and we fix a quadratic form $q$ on $\bQ^{n+1}$.
We denote by $m$ the Witt index of $q$ and by
$Z=Z(q_\bR)\subseteq\bP^n(\bR)$ the corresponding quadratic
hypersurface.  We also assume that $\Zli$ is not empty, and
define
\[
 E=\bP^n(\bQ)\setminus Z
 \et
 S=\{\ux\in\bZ^{n+1}\,;\, q(\ux)\neq 0\}
  =\{\ux\in\bZ^{n+1}\setminus\{0\}\,;\, [\ux]\in E\}.
\]
Moreover, we assume that $q$ is integer valued on $\bZ^{n+1}$ upon
multiplying it by a suitable positive integer if necessary
(this does not affect the hypersurface $Z$ nor the Witt index $m$).
Finally, we set $\rho=\rho_n$ (as defined in Theorem \ref{intro:thm}).
We start with the following simple but crucial observation.

\begin{lemma}
\label{outside:lemma}
Let $\uxi\in\bR^{n+1}$ be a representative of a point $\xi\in Z$.
Then, there is a constant $b>0$ such that $\cL_\uxi(X;S)\ge bX^{-1}$
for each sufficiently large $X$.
\end{lemma}

\begin{proof}
For each $\ux\in S$, we have $|q(\ux)|\ge 1$ since $q(\ux)$ is
a non-zero integer. By Lemma \ref{est:lemma}, we also have
$|q(\ux)|\le b^{-1}\norm{\ux}\cL_\uxi(\ux)$ for a constant $b>0$
depending only on $\uxi$ and $q$.  Thus $L_\uxi(\ux)\ge
b\norm{\ux}^{-1}$ for each $\ux\in S$ and so
$\cL_\uxi(X;S)\ge bX^{-1}$ for each $X\ge 1$.
\end{proof}

We can now prove Theorem \ref{results:thm2} (i).

\begin{proposition}
\label{outside:prop1}
For each $\xi\in\Zli$, we have
$\disp\limsup_{X\to\infty} X^{1/\rho}\cD_\xi(X;E) > 0$.
\end{proposition}

\begin{proof}
Let $\uxi\in\bR^{n+1}$ be a representative of $\xi$.
If $\cL_\uxi(X;S) > X^{-1/\rho}$ for arbitrarily large values
of $X$, then $\limsup X^{1/\rho}\cL_\uxi(X;S) > 0$
and we are done since $\cD_\xi(X;E)\asymp\cL_\uxi(X;S)$
(see Section \ref{sec:est}).
Thus, we may assume that $\cL_\uxi(X;S) \le X^{-1/\rho}$
for each sufficiently large $X$.  Then, by the above lemma,
Condition \eqref{NPR:thm:eq1} of Theorem \ref{NPR:thm}
holds with $\alpha=1/\rho$, $\beta=1$, $a=1$ and some $b>0$.
By definition of $\rho=\rho_n$, the corresponding value for
$\epsilon$ is $1- (\alpha + \alpha^2 + \cdots + \alpha^n)=0$.
Thus Corollary \ref{NPR:cor} yields once again that
$\limsup X^{1/\rho}\cL_\uxi(X;S)>0$.
\end{proof}

Applying Theorem \ref{NPR:thm}, we also deduce the following
statement.

\begin{proposition}
\label{outside:prop2}
For each sufficiently small $\eta>0$, there exists $\delta>0$ with
the following property.  If a point $\xi\in\Zli$ satisfies
\begin{equation}
 \label{outside:prop2:eq1}
 \limsup_{X\to\infty}  X^{1/\rho-\delta}\cD_\xi(X;E) < \infty
\end{equation}
then there is an unbounded sequence $(\uy_i)_{i\geq 0}$ of primitive
integer points in $S$ which for each sufficiently large index $i$ satisfy
 \begin{itemize}
  \item[(i)]
   $\norm{\uy_{i}}^{\rho-\eta}
     \le \norm{\uy_{i+1}}
     \le \norm{\uy_{i}}^{\rho+\eta}$;
  \smallskip
  \item[(ii)]
   $D_\xi(\uy_i) \le \norm{\uy_{i}}^{-1+\eta}$;
  \smallskip
  \item[(iii)] $\det(\uy_i,\dots,\uy_{i+n})\neq 0$;
  \smallskip
  \item[(iv)] $\uy_{i+1}$ is a rational multiple of $\psi(\uy_i,\uy_{i-n})$.
\end{itemize}
\end{proposition}

Assuming that $\eta\in(0,\rho-1)$, Condition (ii) implies that
the sequence of points $[\uy_i]$ in $\bP^n(\bQ)$ converges
to $\xi$ in $\bP^n(\bR)$ because
\[
 \dist(\xi,[\uy_i])
  = \norm{\uy_i}^{-1}D_\xi(\uy_i)
  \le D_\xi(\uy_i),
\]
Moreover Condition (iv) shows that this sequence is uniquely
determined by its first terms since $\psi$ is a bi-homogeneous map.
As there are countably finite sequences in $\bP^n(\bQ)$,
we conclude that there are at most countably many points
$\xi\in\Zli$ which satisfy \eqref{outside:prop2:eq1}
for the corresponding $\delta$ and this proves
Theorem \ref{results:thm2} (iii).

\begin{proof}
Choose an arbitrarily small $\delta\in(0,1/2)$ and assume that a
point $\xi\in\Zli$ with representative $\uxi\in\bR^{n+1}$ satisfies
\eqref{outside:prop2:eq1}.  Then, this assumption together with
Lemma \ref{outside:lemma} implies that Condition \eqref{NPR:thm:eq1}
of Theorem \ref{NPR:thm} holds with $\alpha=1/\rho-\delta$,
$\beta=1$ and some $a,b>0$.  Since the corresponding
$\epsilon = 1- (\alpha + \cdots + \alpha^n)$
vanishes for $\alpha=1/\rho$, Condition \eqref{NPR:thm:eq3}
is also fulfilled if $\delta$ is small enough as a function of $n$
alone.  Then, the theorem provides an unbounded sequence
$(\uy_i)_{i\geq 0}$ of integer points in $S$ which are primitive
(because of their property (iv) of minimality) and which
satisfy the above conditions (i) to (iii)
with $\eta=\cO_n(\delta)$ for each
sufficiently large $i$, say for each $i\ge i_0$.

Finally, let $\uz=\psi(\uy_{i},\uy_{i+1})$ for some
$i\ge 2n+i_0$.  Using Lemma \ref{est:lemma}, we find that
\begin{align*}
 L_\uxi(\uz)
  &\ll \norm{\uy_{i+1}}L_\uxi(\uy_i)^2
       + \norm{\uy_i}L_\uxi(\uy_i)L_\uxi(\uy_{i+1})
   \ll \norm{\uy_i}^{\rho-2+\cO_n(\delta)},\\
 \norm{\uz}
  &\ll \norm{\uy_{i+1}}^2L_\uxi(\uy_i)
       + \norm{\uy_i}^{\rho-2+\cO_n(\delta)}
  \ll \norm{\uy_i}^{2-\rho+\cO_n(\delta)},
\end{align*}
with implied constants that are independent of $i$.  Since
$2-\rho=1/\rho^n$, this means that, for $j=i-n$, we have
\[
 \norm{\uz}\le \norm{\uy_j}^{1+\cO_n(\delta)}
 \et
 L_\uxi(\uz)
   \le \norm{\uy_j}^{-1+\cO_n(\delta)}
    = L_\uxi(\uy_j)^{1+\cO_n(\delta)}.
\]
Now, suppose that $\uz$ is not a multiple of $\uy_j$.  Then
$\uz$ and $\uy_j$ are linearly independent over $\bQ$ and so
there exists an integer $k$ with $1\le k\le n$ such that
$\uy_{j-n},\dots,\widehat{\uy_{j-k}},\dots,\uy_j,\uz$
form a basis of $\bQ^{n+1}$.  Since $\uy_j$ has the largest norm
and yields the smallest value for the function $L_\uxi$ among the
points $\uy_{j-n},\dots,\uy_j$, we find
\begin{align*}
 1\le |\det(\uy_{j-n},&\dots,\widehat{\uy_{j-k}},\dots,\uy_j,\uz)|\\
 &\ll
      \big(\norm{\uz}L_\uxi(\uy_{j})+L_\uxi(\uz)\norm{\uy_{j}}\big)
        L_\uxi(\uy_{j-n}) \cdots \widehat{L_\uxi(\uy_{i-k})} \cdots L_\uxi(\uy_{j-1})\\
  &\ll
      \norm{\uy_j}^{\cO_n(\delta)-(\rho^{-n}+\cdots+\rho^{-2})}.
\end{align*}
Assuming $\delta$ small enough as a function of $n$ alone,
this yields an upper bound on $\norm{\uy_j}$ and thus on $i$.
Then Condition (iv) is fulfilled as well.
\end{proof}

%
%

\section{Construction of extremal points}
\label{sec:construction}

We now turn to the proof of Theorem \ref{results:thm2} (ii), so $n\ge 2$.
As mentioned after Lemma \ref{equiv:lemma}, we may assume
that the hypersurface $Z$ of $\bP^{n+1}(\bR)$ is defined by
the quadratic form $q\colon\bQ^{n+1}\to\bQ$ given by
\begin{equation}
 \label{construction:eq:choix1}
 q(t_0,\dots,t_n)=t_0^2-a_1t_1^2-\cdots-a_{n}t_{n}^2
\end{equation}
for integers $a_1,\dots,a_{n}$ where $a_1>0$ is not
a square, or by
\begin{equation}
 \label{construction:eq:choix2}
 q(t_0,\dots,t_n)=t_0t_1-a_2t_2^2-\cdots-a_{n}t_{n}^2
\end{equation}
for integers $a_2,\dots,a_{n}$ with $a_2\neq0$. We need to
show the existence of infinitely many points $\xi$ in $\Zli$
such that
\begin{equation}
 \label{construction:eq:limsup}
 \limsup_{X\to\infty} X^{1/\rho}\cD_\xi(X;E)<\infty
\end{equation}
where $E=\bP^n(\bQ)\setminus Z$ and where $\rho=\rho_n$ is
as in the statement of Theorem \ref{intro:thm}.  However,
showing the existence of a single point $\xi$ suffices because if $\xi$
has this property, then it follows from Lemma \ref{equiv:lemma}
that $T(\xi)$ shares the same property for any automorphism
$T\in\GL_{n+1}(\bQ)$ such that $q\circ T=q$.  Indeed, this
group of automorphisms of $q$ is infinite and we have $T(\xi)=\xi$
if and only if $T=\pm I$.  For example, if $q$ is given by
\eqref{construction:eq:choix1}, we obtain an
automorphim $T$ of infinite order by choosing a solution
$(u,v)\in\bZ^2$ of the Pell equation $u^2-a_1v^2=1$
with $v\neq 0$ and by defining
\[
 T(t_0,t_1,\dots,t_n)=(ut_0+a_1vt_1,vt_0+ut_1,t_2,\dots,t_n)
\]
(this corresponds to multiplication by $u+v\sqrt{a_1}$ in
$\bQ(\sqrt{a_1})$ via the natural isomorphism between
$\bQ^{n+1}$ and $\bQ(\sqrt{a_1})\times \bQ^{n-1}$).
If $q$ is given by \eqref{construction:eq:choix2}, the
automorphim $T$ given by
\[
 T(t_0,t_1,\dots,t_n)=(2t_0,t_1/2,t_2,\dots,t_n)
\]
is also of infinite order.

Thus we simply need to construct one point $\xi\in\Zli$
with the property \eqref{construction:eq:limsup}.  We achieve
this through the following result.

\begin{theorem}
\label{construction:thm}
There exists an unbounded sequence of points $(\ux_i)_{i\ge 0}$
in $\bZ^{n+1}$ which upon setting $X_i=\|\ux_i\|$ satisfies
for each $i\ge 0$
\begin{itemize}
\item[(i)] $\ux_{i+n+1}=\psi(\ux_{i+n},\ux_i)$,
\item[(ii)] $q(\ux_i)=1$,
\item[(iii)] $|\det(\ux_{i},\dots,\ux_{i+n})|
    =|\det(\ux_{0},\dots,\ux_{n})|\neq 0$,
\item[(iv)] $X_{i+1}\asymp X_i^{\rho}$,
\item[(v)] $\|\ux_{i+1}\wedge\ux_i\| \asymp X_{i+1}/X_i$,
\end{itemize}
with implied constants that are independent of $i$.  Its image
$\big([\ux_i]\big)_{i\ge 0}$ in $\bP^n(\bR)$ converges to a point
$\xi$ in $\Zli$ such that $D_\xi(\ux_i)\asymp X_i^{-1}$.
Moreover, this point $\xi$ satisfies \eqref{construction:eq:limsup}.
\end{theorem}

Arguing as in the proof of Proposition \ref{outside:prop2},
one may show that, except for the very precise form of
Conditions (i) to (iii), the existence of such a sequence
$(\ux_i)_{i\ge 0}$ is forced upon if we assume that
a point $\xi\in\Zli$ satisfies \eqref{construction:eq:limsup}.
The theorem shows the converse.

The proof of Theorem \ref{construction:thm} requires two main
steps.  We first construct linearly independent points
$\ux_0,\dots,\ux_n$ from the set
\[
 U=\{\ux\in\bZ^{n+1}\,;\,q(\ux)=1\}.
\]
Then we extend them into an infinite sequence $(\ux_i)_{i\ge 0}$
using the recurrence relation (i).  The resulting sequence is
entirely contained in $U$ because Lemma \ref{forms:lemma:psi}
shows that $\psi(\ux,\uy)\in U$ for any $\ux,\uy\in U$.  So,
the recurrence relation (i) simplifies to
\begin{equation}
 \label{eq:rec}
 \ux_{i+n+1}=b(\ux_i,\ux_{i+n})\ux_{i+n}-\ux_i \quad (i\ge 0)
\end{equation}
where $b$ denotes the symmetric bilinear form attached to $q$
(characterized by \eqref{forms:eq:qb}).  In turn this implies
that
\[
 |\det(\ux_{i+1},\dots,\ux_{i+n+1})|
   = |\det(\ux_{i},\dots,\ux_{i+n})|
\]
for each $i\ge 0$.  Thus Conditions (ii) and (iii) are
automatically satisfied.  We show that for a suitable choice of
$\ux_0,\dots,\ux_n$, the asymptotic estimates (iv) and (v) hold
as well, thus proving the first assertion of the theorem.
As we will see, the second assertion follows easily from this.

Before, we go on with the proof, we note that the polynomial
\[
 h(x)=(x-1)(x^n-x^{n-1}-\cdots-x-1)=x^{n+1}-2x^n+1
\]
has only two positive real roots, $1$ and $\rho=\rho_n\in(1,2)$.
Moreover, for any sufficiently small $\epsilon>0$, we have
$2(1+\epsilon)^n>(1+\epsilon)^{n+1}+1$ and so $|2x^n|>|x^{n+1}+1|$
for each $x\in\bC$ with $|x|=1+\epsilon$.  By Rouch\'e's theorem,
this means that $h(x)$ has the same number $n$ of complex roots as
$2x^n$ in the closed disk $|x|\le 1+\epsilon$. Since $\epsilon$
can be taken arbitrarily small, and since $x=1$ is the only root
of $h$ with $|x|=1$, we conclude that besides $1$ and $\rho$,
all roots $x$ of $h$ have $|x|<1$.  In particular, the algebraic integer
$\rho$ is a Pisot number.

%
%

\subsection*{Step 1: Choice of initial points.}

\begin{proposition}
\label{construction:prop:choixU}
There exist a constant $C_0>1$ and points
$\ux_0,\dots,\ux_{n-1}\in U$ with the following property.
For each $B\ge C_0$, there exists $\ux_n\in U$ such that
\begin{equation}
 \label{construction:prop:choixU:eq}
 \|\ux_n\|,\ |\det(\ux_0,\ux_1,\dots,\ux_n)|,\
       |b(\ux_i,\ux_n)|,\  \|\ux_i\wedge\ux_n\|
   \in \Big[\frac{B}{C_0},\, C_0B\Big]
\end{equation}
for $i=1,\dots,n-1$.
\end{proposition}

The most delicate case is when the quadratic form is
given by \eqref{construction:eq:choix1}.  The proof then relies
on a theorem of Lagrange saying that, for any
positive integer $a$, the Pell equation
$x^2-ay^2=1$ admits infinitely many
solutions $(x,y)\in\bZ^2$ if and only if $a$
is not a square.  We will also need the following
consequence of that result.

\begin{lemma}
\label{construction:lemma:trinomial}
Let $a,b\in\bZ$ with $a>0$.  Then the equation $x^2-ay^2-bz^2=1$
admits at least one solution $(x,y,z)\in\bZ^3$ with $xz\neq 0$.
\end{lemma}

\begin{proof}
If $b=0$, we have the solution $(x,y,z)=(1,0,1)$.  Suppose that
$b\neq 0$.  Then the polynomial $at^2+b\in\bZ[t]$ is not the square
of a polynomial of $\bZ[t]$ and so by a statement of P\'olya and
Szeg\"o \cite[Problem 114, p.~132]{PZ1976}, there are arbitrarily large
integers $m$ for which $p(m)=am^2+b$ is not a square (see \cite{Mu2008}
for generalizations).  In the present case,
one can even show that, for any sufficiently large integer $k$, at
least one of the integers $ak^2+b$ or $a(k+1)^2+b$ is not a square.
Fix $m\in\bZ$ such that $c=am^2+b$ is positive but not a square,
and put $y=mz$.  Then the equation becomes $x^2-cz^2=1$
which, by the theorem of Lagrange,
admits a solution $(x,z)\in\bZ^2$ with $xz\neq 0$.
Then $(x,mz,z)$ is a solution of the original equation with
the requested property.
\end{proof}

\begin{proof}[Proof of Proposition \ref{construction:prop:choixU}]
For $i=0,1,\dots,n$, let $\ue_i$ denote the point of $\bZ^{n+1}$
whose $i$-th coordinate is $1$ and all other coordinates are $0$.

Suppose first that $q$ is given by \eqref{construction:eq:choix1}
with $a_1,\dots,a_n\in\bZ$ and $a_1>0$ non-square.
Then, for $i=0,\dots,n-2$, Lemma \ref{construction:lemma:trinomial}
ensures the existence of a point $\ux_i$ in $U$ of the form
\[
 \ux_i=k_i\ue_0+\ell_i\ue_1+m_i\ue_{i+2}
\]
with $(k_i,\ell_i,m_i)\in\bZ^3$ and $k_im_i\neq 0$.
By Lagrange's theorem, there exists also an infinite set
of points of $U$ of the form
\[
 \ux_{n-1}=k_{n-1}\ue_0+\ell_{n-1}\ue_1
\]
with integers $k_{n-1},\ell_{n-1}\ge 1$.  Fix such a choice
of $\ux_0,\dots,\ux_{n-1}\in U$ and put
\[
 \ux_{n}=k\ue_0+\ell\ue_1
\]
for another pair $(k,\ell)$ of positive integers satisfying
$k^2-a_1\ell^2=1$ with $\ell\neq\ell_{n-1}$.
Then $\ux_0,\dots,\ux_n \in U$ are linearly
independent over $\bQ$.  Since $k=\sqrt{a_1}\ell+\cO(1/\ell)$,
we find
\begin{align*}
 &\|\ux_n\| \asymp |\det(\ux_0,\ux_1,\dots,\ux_n)| \asymp \ell,\\
 &\ell
  \gg |b(\ux_i,\ux_n)|
  = 2|kk_i-a_1\ell\ell_i|
  = 2\sqrt{a_1}\ell|k_i-\sqrt{a_1}\ell_i|+\cO(1/\ell) \gg \ell,\\
 &\ell
  \gg \|\ux_i\wedge\ux_n\|
  \ge |k\ell_i-\ell k_i|
  = \ell|k_i-\sqrt{a_1}\ell_i|+\cO(1/\ell) \gg \ell
\end{align*}
for $i=0,\dots,n-1$.  The conclusion follows
because the admissible values of $\ell$ have exponential
growth (they form a linear recurrence sequence).

Suppose now that $q$ is given by \eqref{construction:eq:choix2}
with $a_2,\dots,a_{n-1}\in\bZ$ and $a_2\neq 0$.  For a given
integer $\ell\ge 2$, we choose
\[
 \ux_i=
 \begin{cases}
  (a_{i+2}+1)\ue_0+\ue_1+\ue_{i+2} &\text{for $i=0,\dots,n-2$,}\\
  \ue_0+\ue_1 &\text{if $i=n-1$,}\\
  (a_2\ell^2+1)\ue_0+\ue_1+\ell\ue_2 &\text{if $i=n$.}
 \end{cases}
\]
Then $\ux_0,\dots,\ux_n\in U$ are linearly
independent over $\bQ$.  Moreover, as functions of $\ell$, the numbers
$\|\ux_n\|$, $|\det(\ux_0,\ux_1,\dots,\ux_n)|$,
$|b(\ux_i,\ux_n)|$ and $\|\ux_i\wedge\ux_n\|$
for $i=0,\dots,n-1$ are all equal to $|a_2|\ell^2+\cO(\ell)$
or $\sqrt{2}|a_2|\ell^2+\cO(\ell)$.
The conclusion follows by varying $\ell$.
\end{proof}

%
%

\subsection*{Step 2: Asymptotic estimates}
We will use the approximation lemma of Appendix \ref{sec:appA}
to prove the following statement.

\begin{proposition}
\label{construction:prop:asymp}
Let $C_0$ and $\ux_0,\dots,\ux_{n-1}\in U$ be as
in Proposition \ref{construction:prop:choixU}
for the given quadratic form $q$.  For each sufficiently large $B$
with $B\ge C_0$, the point $\ux_n\in U$ provided
by Proposition \ref{construction:prop:choixU}
has the following property.  The sequence $(\ux_i)_{i\ge 0}$ built on
$(\ux_0,\dots,\ux_n)$ using the recurrence formula
\begin{equation}
\label{etape2:eq:rec}
 \ux_{i+1}
   = \psi(\ux_{i},\ux_{i-n})
   = b(\ux_{i-n},\ux_{i})\ux_{i}-\ux_{i-n} \quad (i\ge n).
\end{equation}
is contained in $U$ and satisfies
\[
 \norm{\ux_{i+1}}\asymp\norm{\ux_i}^\rho
 \et
 \norm{\ux_i\wedge\ux_{i+1}}\asymp\frac{\norm{\ux_{i+1}}}{\norm{\ux_i}}.
\]
\end{proposition}

\begin{proof}
Fix a choice of $B\ge C_0$ and of a corresponding point
$\ux_n\in U$ as in Proposition \ref{construction:prop:choixU}, and consider
the sequence $(\ux_i)_{i\ge 0}$ given by \eqref{etape2:eq:rec}.
By Lemma \ref{forms:lemma:psi}, this sequence is contained
in $U$.  By linearity, the formula \eqref{etape2:eq:rec} yields
\begin{align*}
 b(\ux_{i+1-j},\ux_{i+1})
   &= b(\ux_{i-n},\ux_{i})b(\ux_{i+1-j},\ux_{i})
      - b(\ux_{i-n},\ux_{i+1-j}),\\
 \ux_{i+1-j}\wedge\ux_{i+1}
   &= b(\ux_{i-n},\ux_{i})\ux_{i+1-j}\wedge\ux_{i}
      + \ux_{i-n}\wedge\ux_{i+1-j}
\end{align*}
for each choice of integers $i$ and $j$ with $1\le j\le n\le i$.
Putting
\[
 b_i^{(j)}=b(\ux_{i-j},\ux_i)
 \et
 \uy_i^{(j)}=\ux_{i-j}\wedge\ux_i
\]
these equalities become
\begin{equation}
 \label{etape2:eq:by}
 b_{i+1}^{(j)}=b_i^{(n)}b_i^{(j-1)}-b_{i+1-j}^{(n+1-j)}
 \et
 \uy_{i+1}^{(j)}=b_i^{(n)}\uy_i^{(j-1)}+\uy_{i+1-j}^{(n+1-j)}
 \quad
 (1\le j\le n\le i).
\end{equation}
In particular, for $j=1$, they simplify to
\begin{equation}
 \label{etape2:eq:by1}
 b_{i+1}^{(1)}=b_i^{(n)}
 \et
 \uy_{i+1}^{(1)}=\uy_{i}^{(n)}
 \quad
 (i\ge n),
\end{equation}
since $b_i^{(0)}=b(\ux_i,\ux_i)=2q(\ux_i)=2$ and
$\uy_i^{(0)}=\ux_i\wedge\ux_i=0$.  Put also
\begin{align*}
 B_i&=\max\{|b_k^{(j)}|\,;\, 0\le k\le i
      \ \ \text{and}\ \ 0\le j \le \min\{n,k\}\},\\
 Y_i&=\max\{ \norm{\uy_k^{(j)}} \,;\, 0\le k\le i
      \ \ \text{and}\ \ 0\le j \le \min\{n,k\}\},\\
 X_i&=\max\{ \norm{\ux_k} \,;\, 0\le k\le i\}
\end{align*}
for each $i\ge 0$. For $i\ge n$ and $j\in\{2,\dots,n\}$,
the formulas \eqref{etape2:eq:by} imply that
\begin{align}
 |b_i^{(n)}|\,|b_i^{(j-1)}|-B_{i-1}&\le
   |b_{i+1}^{(j)}|\le |b_i^{(n)}|\,|b_i^{(j-1)}|+B_{i-1},
 \label{etape2:eq:bornesb}\\
 |b_i^{(n)}|\,\norm{\uy_i^{(j-1)}} - Y_{i-1}
  &\le \norm{\uy_{i+1}^{(j)}}
  \le |b_i^{(n)}|\,\norm{\uy_i^{(j-1)}}+Y_{i-1},
 \label{etape2:eq:bornesy}
\end{align}
while the recurrence formula \eqref{etape2:eq:rec} yields
\begin{equation}
 \label{etape2:eq:X}
 |b_i^{(n)}|\,\norm{\ux_i} - X_{i-1}
  \le \norm{\ux_{i+1}}
  \le |b_i^{(n)}|\,\norm{\ux_i} + X_{i-1}.
\end{equation}

We first show that, if $B$ is large enough,
the numbers $|b_i^{(j)}|$ and $\norm{\uy_i^{(j)}}$
behave like powers of $B$.  To this end, consider
the integers $m(i,j)$ defined recursively for
$i\ge 0$ and $j=1,\dots,n$ by
\[
 m(i,j) =
 \begin{cases}
  0 &\text{if $0\le i< n$ and $1\le j\le n$,}\\
  1 &\text{if $i=n$ and $1\le j\le n$,}\\
  m(i-1,n) &\text{if $i>n$ and $j=1$,}\\[2pt]
  m(i-1,n)+m(i-1,j-1) &\text{if $i>n$ and $2\le j\le n$.}
 \end{cases}
\]
It is relatively easy to prove that they satisfy
\begin{align}
 m(i,1) \le m(i,2) \le \cdots \le m(i,n) \quad (i\ge 0),
 \label{etape2:eq:m1}\\
 \frac{3}{2}m(i,n) \le m(i+1,n) \le 2m(i,n) \quad (i\ge n).
 \label{etape2:eq:m2}
\end{align}
In particular $m(i,n)$ tends to infinity with $i$.
Define also
\[
 C_n=\max\{C_0,B_{n-1},X_{n-1},Y_{n-1}\}
 \et
 C_{i+1}=2^{1/m(i,n)}C_i \quad (i\ge n).
\]
These numbers are independent of $B$ since $B_{n-1}$, $X_{n-1}$
and $Y_{n-1}$ are functions of $\ux_0,\dots,\ux_{n-1}$ only.
By \eqref{etape2:eq:m2}, we have $m(i,n)\ge (3/2)^{i-n}$ for
each $i\ge n$, thus $\sum_{i=n}^\infty 1/m(i,n)\le 3$ and so
the sequence $(C_i)_{i\ge n}$ is bounded above by $C:=8C_n$.

Assume from now on that $B\ge C^5$.  We claim that for each
$i\ge n$ we have
\begin{align}
 &X_i=\norm{\ux_i}>X_{i-1},
   \label{etape2:eq:Xmin}\\
 &\max\{B_{i-1},Y_{i-1}\}
   \le \frac{1}{2}\Big(\frac{B}{C_i}\Big)^{m(i,n)},
   \label{etape2:eq:BY}\\
 &\min\{|b_i^{(j)}|, \norm{\uy_i^{(j)}}\}
   \ge \Big(\frac{B}{C_i}\Big)^{m(i,j)} \
   \quad (1\le j\le n),
   \label{etape2:eq:BYmin}\\
 &\max\{|b_i^{(j)}|, \norm{\uy_i^{(j)}}\}
   \le (C_i B)^{m(i,j)} \
   \quad (1\le j\le n).
   \label{etape2:eq:BYmaj}
\end{align}

Arguing by induction on $i$, suppose first that $i=n$.
Then we have $m(i,j)=1$ for all $j$. So \eqref{etape2:eq:BYmin}
and \eqref{etape2:eq:BYmaj} follow immediately from the choice of $\ux_n$
(this only requires $B\ge C_0$).  Since $\norm{\ux_n}\ge B/C_0$
and $X_{n-1}\le C_n$, we also have $\norm{\ux_n}>X_{n-1}$, as
$B>C_0C_n$.  This yields \eqref{etape2:eq:Xmin}. Finally
\eqref{etape2:eq:BY} follows from $\max\{B_{n-1},Y_{n-1}\}\le C_n$,
as $B\ge 2C_n^2$.

Suppose now that \eqref{etape2:eq:Xmin} to \eqref{etape2:eq:BYmaj}
hold for some integer $i\ge n$.  Then \eqref{etape2:eq:BYmin}
gives $|b_i^{(n)}|\ge B/C_i\ge 2$.  Thus \eqref{etape2:eq:X}
together with \eqref{etape2:eq:Xmin} yield
\[
 \norm{\ux_{i+1}}\ge 2X_i-X_{i-1}>X_i,
\]
so $X_{i+1}=\norm{\ux_{i+1}}>X_i$.  Because
of \eqref{etape2:eq:m1} and \eqref{etape2:eq:BY}, the
inequalities \eqref{etape2:eq:BYmaj} imply
\begin{equation}
 \label{etape2:eq:maxBYi}
 \max\{B_{i},Y_{i}\} \le (C_iB)^{m(i,n)}.
\end{equation}
Since $m(i+1,n)\ge (3/2)m(i,n)$ and $B\ge C_{i+1}^5$, we deduce that
\[
 \Big(\frac{B}{C_{i+1}}\Big)^{m(i+1,n)}
  \ge \Big(\frac{B}{C_{i+1}^5}\Big)^{m(i,n)/2}(C_{i+1}B)^{m(i,n)}
  \ge 2 \max\{B_i,Y_i\},
\]
which proves \eqref{etape2:eq:BY} with $i$ replaced by $i+1$.
Using the hypothesis \eqref{etape2:eq:BYmin}, the
equalities \eqref{etape2:eq:by1} yield
\[
 \min\{|b_{i+1}^{(1)}|, \norm{\uy_{i+1}^{(1)}}\}
  = \min\{|b_{i}^{(n)}|, \norm{\uy_{i}^{(n)}}\}
  \ge \Big(\frac{B}{C_i}\Big)^{m(i,n)}
  \ge \Big(\frac{B}{C_{i+1}}\Big)^{m(i+1,1)},
\]
while using \eqref{etape2:eq:BY} and \eqref{etape2:eq:BYmin},
the estimates \eqref{etape2:eq:bornesb} and \eqref{etape2:eq:bornesy}
imply for $j=2,\dots,n$
\[
 \begin{aligned}
 \min\{|b_{i+1}^{(j)}|,\norm{\uy_{i+1}^{(j)}}\}
  &\ge |b_i^{(n)}|\min\{|b_{i}^{(j-1)}|,\norm{\uy_{i}^{(j-1)}}\}
        - \max\{B_{i-1},Y_{i-1}\}\\
  &\ge \Big(\frac{B}{C_i}\Big)^{m(i,n)+m(i,j-1)}
      - \frac{1}{2}\Big(\frac{B}{C_i}\Big)^{m(i,n)} \\
  &\ge \frac{1}{2}\Big(\frac{B}{C_i}\Big)^{m(i+1,j)} \\
  &\ge \Big(\frac{B}{C_{i+1}}\Big)^{m(i+1,j)}.
 \end{aligned}
\]
This proves \eqref{etape2:eq:BYmin} with $i$ replaced by $i+1$,
We omit the proof of the induction step for \eqref{etape2:eq:BYmaj}
as it is similar.  In that case, we may even replace the use of
\eqref{etape2:eq:BY} by the weaker estimate $\max\{B_{i-1},Y_{i-1}\}
\le (C_iB)^{m(i,n)}$ which follows from \eqref{etape2:eq:maxBYi}.

So, under the current hypothesis $B\ge C^5$, we have
\begin{align*}
 &X_i=\norm{\ux_i}>X_{i-1},
 \quad
 \max\{B_{i-1},Y_{i-1}\}\le \frac{1}{2}\Big(\frac{B}{C_i}\Big)^{m(i,n)},
 \quad
 \min\{|b_i^{(j)}|, \norm{\uy_i^{(j)}}\}
   \ge \Big(\frac{B}{C_i}\Big)^{m(i,j)}
\end{align*}
whenever $1\le j\le n\le i$.  In particular, this yields the crude estimate
\[
 \max\{B_{i-1},Y_{i-1}\} \le |b_i^{(n)}|.
\]
Thus, for each $i\ge n$ and $j=2,\dots,n$, the
inequalities \eqref{etape2:eq:bornesb}, \eqref{etape2:eq:bornesy}
and \eqref{etape2:eq:X} imply that
\begin{align*}
 |b_{i+1}^{(j)}|&=|b_i^{(n)}|\,|b_i^{(j-1)}|(1+\epsilon_{i,j})
 &&\text{with} \quad
   |\epsilon_{i,j}|
     \le \frac{B_{i-1}}{|b_i^{(n)}|\,|b_i^{(j-1)}|}
     \le \Big(\frac{C}{B}\Big)^{m(i,j-1)},\\
 \norm{\uy_{i+1}^{(j)}}&=|b_i^{(n)}|\,\norm{\uy_i^{(j-1)}}(1+\epsilon'_{i,j})
 &&\text{with} \quad
   |\epsilon'_{i,j}|
     \le \frac{Y_{i-1}}{|b_i^{(n)}|\,\norm{\uy_i^{(j-1)}}}
     \le \Big(\frac{C}{B}\Big)^{m(i,j-1)},\\
 X_{i+1}&=|b_i^{(n)}|\, X_i(1+\epsilon''_{i})
 &&\text{with} \quad
   |\epsilon''_{i}|
     \le \frac{X_{i-1}}{|b_i^{(n)}|\,X_i}
     \le \Big(\frac{C}{B}\Big)^{m(i,n)}.
\end{align*}
For $j=1$, we have instead the formulas \eqref{etape2:eq:by1}.
We conclude that, for each $i\ge n$, the vector
\[
 \uv_i:=\big(
        \log|b_i^{(1)}|,\dots,\log|b_i^{(n)}|,
        \log\norm{\uy_i^{(1)}},\dots,\log\norm{\uy_i^{(n)}},
        \log X_i
        \big)
        \in \bR^{2n+1},
\]
satisfies
\[
 \norm{\uv_{i+1}-T(\uv_i)} \ll \Big(\frac{C}{B}\Big)^{m(i,1)},
\]
where $T\colon\bR^{2n+1}\to\bR^{2n+1}$ is the linear operator
given by
\begin{align*}
 T(x_1,&\dots,x_{n},y_1,\dots,y_{n},z)\\
 &=(x_{n},x_1+x_{n},\dots,x_{n-1}+x_{n},
   y_{n},y_1+x_{n},\dots,y_{n-1}+x_{n},
   z+x_{n}).
\end{align*}
We note that the matrix of $T$ in the canonical basis is
lower triangular by blocks with three blocks on the
diagonal, namely the companion matrices of the
polynomials
\[
 p(x)=x^{n}-x^{n-1}-\cdots-x-1,
 \quad
 q(x)=x^{n}-1
 \et
 r(x)=x-1.
\]
Thus the characteristic polynomial of $T$ is
$p(x)q(x)r(x)$.  We also observe that $x=1$ is the only
multiple root of this product and that it is a double
root (this follows for example from
$p(x)r(x)=x^{n+1}-2x^{n}+1=(x-2)q(x)+(x-1)$).  On the
other hand, the eigenspace of $T$ for the eigenvalue
$1$ contains the vectors
\[
 \big(\underbrace{0,\dots,0}_{\textstyle{n}},
  \underbrace{1,\dots,1}_{\textstyle{n}},0\big)
 \et
 \big(\underbrace{0,\dots,0}_{\textstyle{n}},
  \underbrace{0,\dots,0}_{\textstyle{n}},1\big).
\]
So it has dimension 2 and thus the minimal polynomial of $T$
must be $m_T(x)=p(x)q(x)$.  Its roots in $\bC$ are all
simple and, by the comments made before Step 1, its root
$\rho=\rho_n$ is the only one of absolute value $>1$.  Moreover
\[
 \uv=\Big(1-\frac{1}{\rho},\dots,1-\frac{1}{\rho^{n}},
          1-\frac{1}{\rho},\dots,1-\frac{1}{\rho^{n}},1\Big)
\]
is an eigenvector of $T$ for $\rho=2-1/\rho^n$.
As $\sum_i(C/B)^{m(i,1)}<\infty$,
the approximation lemma \ref{lemme:approx} in the appendix yields a
constant $\alpha\in\bR$ such that the differences
$\uv_i-\alpha\rho^i\uv$ are bounded.  This means that for
each pair of integers $(i,j)$ with $1\le j \le n\le i$ we
have
\[
 |b_i^{(j)}| \asymp \norm{\uy_i^{(j)}} \asymp \exp(\alpha(\rho^i-\rho^{i-j}))
 \et
 X_i \asymp \exp(\alpha\rho^i),
\]
so $X_{i+1}\asymp X_i^\rho$ and
$\norm{\ux_i\wedge\ux_{i+1}}=\norm{\uy_{i+1}^{(1)}}
\asymp X_{i+1}/X_i$.
\end{proof}

\subsection*{Proof of Theorem \ref{construction:thm}}
Proposition \ref{construction:prop:asymp} provides a
sequence $(\ux_i)_{i\ge 0}$ in $U$ which satisfies the
five conditions (i) to (v) of the theorem.  By (v),
we have
\[
 \dist([\ux_i],[\ux_{i+1}])
  =\frac{\norm{\ux_i\wedge\ux_{i+1}}}{X_iX_{i+1}}
  \asymp X_i^{-2},
\]
and, by (iv), $X_i$ goes to infinity faster than any geometric
sequence. Thus $[\ux_i]$ converges to a point
$\xi\in\bP^n(\bR)$ with
\[
 D_\xi(\ux_i)=X_i\dist([\ux_i],\xi) \asymp X_i^{-1}.
\]
Choose a representative $\uxi\in\bR^{n+1}$ with $\norm{\uxi}=1$.  As
$\pm\/X_i^{-1}\ux_i$ converges to $\uxi$ for an appropriate choice of
signs, and as $q(\pm X_i^{-1}\ux_i)=X_i^{-2}$ converges to $0$,
we find that $q(\uxi)=0$, thus $\xi\in Z$. By Lemma \ref{est:lemma:li},
we deduce that $\xi\in\Zli$ because we have
$\lim_{i\to\infty}D_\xi(\ux_i)=0$ and the crucial property (iii) gives
$\Qspan{\ux_i,\dots,\ux_{i+n}}=\bQ^{n+1}$ for each $i\ge 0$. Finally,
for each sufficiently large $X$, there exists an index $i\ge 0$
such that $X_i\le X <X_{i+1}$ and, since $[\ux_i]\in E$, we
obtain
\[
 X^{1/\rho}\cD_\xi(X;E)
  \ll X^{1/\rho} D_\xi(\ux_i)
  \ll X_{i+1}^{1/\rho} X_i^{-1}
  \ll 1,
\]
showing that $\limsup_{X\to\infty}X^{1/\rho}\cD_\xi(X;E)<\infty$,
as announced.

%
%

\section{Quadratic forms of Witt index at most $1$}
\label{sec:proofiv}

In this section, we simply assume that $n\ge 1$ and we fix a
quadratic form $q$ on $\bQ^{n+1}$.  Our goal is to prove
Theorem \ref{results:thm2} (iv).  We start with a general
estimate which compares to \cite[Lemme Clef]{GR2017} for
isotropic subspaces of dimension $1$.

\begin{lemma}
 \label{lemma:xyz}
Suppose that $\ux$ and $\uy$ are linearly independent points
of\/ $\bQ^{n+1}$, that the subspace $W$ of\/ $\bQ^{n+1}$ that they
span is not totally isotropic, and that $q(\uy)=0$.
Put $\uz=\psi(\ux,\uy)$.  Then $\uz$ is non-zero with $q(\uz)=0$,
 and we have
\[
 \|\uy\|\,\|\uz\|\le 2\|q\|\, \|\ux\wedge\uy\|^2
\]
where $\|q\|:=\max\{|q_\bR(\ux)|\,;\,\|\ux\|=1\}$.
\end{lemma}

\begin{proof}
As $\{\ux,\uy\}$ is a basis of $W$ and as $W$ is not totally isotropic,
we have $b(\ux,\uy)\neq 0$ or $q(\ux)\neq 0$, and so
$\uz=b(\ux,\uy)\ux-q(\ux)\uy$ is nonzero.  We also have
$q(\uz)=q(\ux)^2q(\uy)=0$ by Lemma \ref{forms:lemma:psi}.  Since the angle between
the lines spanned by $\uy$ and $\uz$ in $\bR^{n+1}$ is at most $\pi/2$,
there is a vector $\uu\in\bR^{n+1}$ of norm $1$ which makes angles of
at most $\pi/4$ with each of those lines.  This means that
\[
 |\uu\cdot\uy|\ge \frac{1}{\sqrt{2}}\|\uy\|
 \et
 |\uu\cdot\uz|\ge \frac{1}{\sqrt{2}}\|\uz\|
\]
where the dot represents the standard scalar product in $\bR^{n+1}$.  The point
\[
 \uw:=(\uu\cdot\uy)\ux-(\uu\cdot\ux)\uy
\]
obtained by contraction of $\ux\wedge\uy$ with $\uu$ has norm
$\|\uw\|\le \|\uu\|\,\|\ux\wedge\uy\| = \|\ux\wedge\uy\|$
by \eqref{est:eq:contraction}.  Since $q(\uy)=0$, we find that
\[
 q_\bR(\uw)=(\uu\cdot\uy)^2q(\ux)-(\uu\cdot\ux)(\uu\cdot\uy)b(\ux,\uy)
       =-(\uu\cdot\uy)(\uu\cdot\uz).
\]
Altogether, this yields
\[
 \frac{1}{2}\|\uy\|\,\|\uz\|
   \le |q_\bR(\uw)|
   \le \|q\|\,\|\uw\|^2
   \le \|q\|\,\|\ux\wedge\uy\|^2.
\qedhere
\]
\end{proof}

\begin{cor}
\label{proofiv:height}
Suppose further that $\ux,\uy\in\bZ^{n+1}$, then we have
$\norm{\uy}\le c\norm{\ux\wedge\uy}^2$
for a constant $c>0$ depending only on $q$.
\end{cor}

\begin{proof}
Choose an integer $m\ge 1$ such that
$mq(\bZ^{n+1})\subseteq\bZ$.  Then $m\uz$ is a non-zero
integer point, so $\norm{\uz}\ge 1/m$ and therefore
$\norm{\uy}\le 2m\norm{q}\norm{\ux\wedge\uy}^2$.
\end{proof}

We can now state and prove the main result of this section.
In a corollary below, we will show that it implies
Theorem \ref{results:thm2} (iv).

\begin{proposition}
\label{proofiv:prop}
Let $Z=Z(q_\bR)\subseteq\bP^n(\bR)$.
Suppose that $q$ has Witt index $m\le 1$ and that
a point $\xi\in\Zli$ has $\lambdahat(\xi)>1/2$. Then, we have
$D_\xi(\ux)\gg \norm{\ux}^{-1/2}$ for each
non-zero $\ux\in\bZ^{n+1}$ with $q(\ux)=0$.
\end{proposition}

\begin{proof}
Suppose first that $q$ is non-degenerate.  Then, the
maximal totally isotropic subspaces of $\bQ^{n+1}$
have dimension $m\le 1$.  Choose a
representative $\uxi$ of $\xi$ in $\bR^{n+1}$ with
$\norm{\uxi}=1$, a real number $\lambda$ with
$1/2<\lambda<\lambdahat(\xi)$, and $X_0\ge 1$ such
that $\cD_\xi(X)\le X^{-\lambda}$ for each $X\ge X_0$.
Suppose that $\uy$ is a primitive point of $\bZ^{n+1}$
with $q(\uy)=0$ and $\norm{\uy}\ge 2X_0$.  There
exists a non-zero point $\ux\in\bZ^{n+1}$ with
$\norm{\ux} \le \norm{\uy}/2$ and $D_\xi(\ux)
= \norm{\ux\wedge\uxi} \le (\norm{\uy}/2)^{-\lambda}$.
Then $\langle\ux,\uy\rangle_\bQ$ is a subspace
of $\bQ^{n+1}$ of dimension $2$, and so it is not
totally isotropic.  Since $q(\uy)=0$, Corollary
\ref{proofiv:height} gives
\[
 \norm{\uy} \le c\norm{\ux\wedge\uy}^2
\]
with $c=c(q)>0$. On the other hand, the triangle
inequality $\dist(\ux,\uy)\le \dist(\ux,\uxi)+\dist(\uy,\uxi)$
yields
\[
 \norm{\ux\wedge\uy}
  \le \norm{\ux}\,\norm{\uy\wedge\uxi}
      + \norm{\uy}\,\norm{\ux\wedge\uxi}
  \le \frac{1}{2}\norm{\uy}\,\norm{\uy\wedge\uxi}
      + 2^\lambda\,\norm{\uy}^{1-\lambda}.
\]
Altogether, this implies that
\[
 \norm{\uy\wedge\uxi}
 \ge 2\big(c^{-1/2}\norm{\uy}^{1/2}
           -2^\lambda\norm{\uy}^{1-\lambda}\big)
      \norm{\uy}^{-1}
\]
and thus $D_\xi(\uy)\ge c^{-1/2}\norm{\uy}^{-1/2}$ if
$\norm{\uy}$ is large enough.  We conclude that
$D_\xi(\uy)\gg \norm{\uy}^{-1/2}$ for any primitive
point $\uy$ of $\bZ^{n+1}$, and thus for any non-zero
$\uy\in\bZ^{n+1}$.

In general, $q$ has rank $r+1$ for an integer $r$ with
$1\le r\le n$ (we have $r\ge 1$ since $\Zli\neq\emptyset$).
Choose a representative $\uxi$ of
$\xi$ in $\bR^{n+1}$, and a $\bQ$-linear automorphism
$T$ of $\bQ^{n+1}$ such that $T(\bZ^{n+1})=\bZ^{n+1}$ and
\[
 \{0\}^{r+1}\times\bQ^{n-r} = T^{-1}(\ker(q)) = \ker(q\circ T).
\]
Then there is a non-degenerate quadratic form
$\tq\colon\bQ^{r+1}\to\bQ$ such that
\[
 (q\circ T)(x_0,\dots,x_n)=\tq(x_0,\dots,x_r)
\]
for each $(x_0,\dots,x_n)\in\bQ^{n+1}$.
Set $\ueta=T_\bR^{-1}(\uxi)$ and write
$\ueta=(\eta_0,\dots,\eta_n)$.  Then $\ueta$ has linearly
independent coordinates over $\bQ$ and the point
$\eta=[\ueta]$ satisfies $\lambdahat(\eta)=\lambdahat(\xi)$.
Moreover $\tueta=(\eta_0,\dots,\eta_r)$ is a zero of
$\tq_\bR$.  Thus $\teta=[\tueta]$ belongs to $\tZli$ where
$\tZ=Z(\tq_\bR)$ is the quadratic hypersurface of $\bP^r(\bR)$
associated to $\tq$.  We also have $\lambdahat(\teta)
\ge \lambdahat(\eta)$, thus $\lambdahat(\teta)>1/2$.  By the
above, this implies that $\norm{\tuy\wedge\tueta}\gg
\norm{\tuy}^{-1/2}$ for any non-zero $\tuy\in\bZ^{r+1}$.
Now, let $\ux\in\bZ^{n+1}\setminus\{0\}$.
Set $\uy=T_\bR^{-1}(\ux)$ and write $\uy=(y_0,\dots,y_n)$.
We have $\norm{\uy\wedge\ueta}\asymp\norm{\ux\wedge\uxi}$.
So, if $\norm{\ux\wedge\uxi}$ is small enough, the point
$\tuy=(y_0,\dots,y_r)$ is non-zero with
\[
 \norm{\tuy}\le\norm{\uy}\ll\norm{\ux}
 \et
 \norm{\tuy\wedge\tueta}
   \le \norm{\uy\wedge\ueta}
   \ll \norm{\ux\wedge\uxi},
\]
thus $\norm{\ux\wedge\uxi}\gg \norm{\ux}^{-1/2}$.
\end{proof}

\begin{cor}
Let $E=\bP^n(\bQ)\setminus Z(\bQ)$.  With the hypotheses
of the previous proposition, we have $\cD_\xi(X)=\cD_\xi(X;E)$
for any sufficiently large $X$.
\end{cor}

\begin{proof}
By Proposition \ref{proofiv:prop}, we have
$\cD_\xi(X;Z(\bQ))\gg X^{-1/2}$ for each $X\ge 1$.  Since
$\lambdahat(\xi)>1/2$, this yields $\cD_\xi(X;Z(\bQ)) >
\cD_\xi(X)$ and so $\cD_\xi(X)=\cD_\xi(X;E)$, for
each large enough $X$.
\end{proof}

%
%

\section{Quadratic forms of higher Witt index}
\label{sec:proofv}

In this last section, we assume that $q$ is a quadratic form on
$\bQ^{n+1}$ of Witt index $m\ge 2$.  Then the rank of $q$ is at
least $2m\ge 4$ and so we have $n\ge 3$.
We denote by $K=\ker(q)$ the kernel of $q$, by $Z=Z(q_\bR)$
the associated quadratic hypersurface of $\bP^n(\bR)$, and by
$Z(\bQ)$ the set of rational points of $Z$.  We start by proving
the first assertion of Theorem \ref{results:thm2} (v).

\begin{proposition}
Suppose that a point $\xi\in\Zli$ has $\lambdahat(\xi)>1/\rho_n$.
Then for each sufficiently large $X$ we have
$\cD_\xi(X)=\cD_\xi(X;Z(\bQ))$.
\end{proposition}

\begin{proof}
Set $E=\bP^n(\bQ)\setminus Z$, and choose $\lambda\in\bR$ with
$1/\rho_n<\lambda<\lambdahat(\xi)$.
Then we have $\cD_\xi(X)\le X^{-\lambda}$ for all sufficiently large $X$
while Theorem \ref{results:thm2} (i) gives $\cD_\xi(X;E)> X^{-\lambda}$
for arbitrarily large values of $X$.  Altogether, this means that
$\cD_\xi(X)=\cD_\xi(X;Z(\bQ))$ for arbitrarily large values of $X$.
Now, suppose that $\cD_\xi(Y)=\cD_\xi(Y;Z(\bQ))$ for some $Y\ge 2$.
To conclude it suffices to show that, if $Y$ is large enough,
we also have $\cD_\xi(X)=\cD_\xi(X;Z(\bQ))$ for each $X\in[Y/2,Y]$.

To prove this claim, choose a zero $\uy$ of $q$ in
$\bZ^{n+1}\setminus\{0\}$ with $\norm{\uy}\le Y$ and
$D_\xi(\uy)=\cD_\xi(Y;Z(\bQ))$. Then $\uy$ is a primitive point
of $\bZ^{n+1}$.  Choose also $X\in[Y/2,Y]$ and
a point $\ux\in\bZ^{n+1}\setminus\{0\}$ with $\norm{\ux}\le X$
and $D_\xi(\ux)=\cD_\xi(X)$.  If $\norm{\uy} \le Y/2$, we may
take $\ux=\uy$ and thus $\cD_\xi(X)=\cD_\xi(X;Z(\bQ))$.
So, we may assume that $\norm{\uy} \ge Y/2$.  If
$\Qspan{\ux,\uy}$ has dimension $2$ and is not totally
isotropic, then using Corollary \ref{proofiv:height} we obtain
\[
 \frac{Y}{2}
   \le \norm{\uy}
   \ll \norm{\ux\wedge\uy}^2
   \ll \big(\norm{\ux}D_\xi(\uy)+\norm{\uy}D_\xi(\ux)\big)^2
   \ll Y^{2(1-\lambda)}
\]
and so $Y$ is bounded from above because $\lambda>1/\rho_n>1/2$.
Otherwise, we have $q(\ux)=0$ and so $\cD_\xi(X)=\cD_\xi(X;Z(\bQ))$.
\end{proof}

Under the hypotheses of the proposition, it can also be shown
by a simple adaptation of the proof that, if $(\ux_i)_{i\ge 1}$
is a sequence of minimal points for $\xi$ with respect
to $\bZ^{n+1}$ (as defined for example in \cite[\S 2]{NPR}),
then $\Qspan{\ux_i,\ux_{i-1}}$ is a totally isotropic subspace
of dimension $2$ for each sufficiently large $i$.  Moreover,
since $\xi\in\Zli$, we also have $\ux_i\notin K$ for all
large enough $i$.  We mention this to motivate the construction
that we undertake below to prove the second assertion
of Theorem \ref{results:thm2} (v). However, before we proceed,
we need the following crucial fact (a separate argument will be
required for the case of rank $4$).

\begin{lemma}
\label{proofv:lemma:evitement}
Suppose that $q$ has rank at least $5$.
Let $\ux\in\bQ^{n+1}$ be a zero of $q$ outside of $K$ and let
$V=\langle \ux\rangle_\bQ^\perp$.   Then, for each finite
set of proper subspaces of $\bQ^{n+1}$not containing $V$,
there is a zero of $q$ in $V$ which lies outside of each of
these subspaces.
\end{lemma}

\begin{proof}
Since $\ux$ is a zero of $q$ not in $K$, it is contained
in a hyperbolic plane $H_1$.  Since $q$ has Witt index at
least $2$ and rank $r\ge 5$, we may write
$\bQ^{n+1}=H_1\perp H_2\perp W \perp K$ where $H_2$ is
another hyperbolic plane and $W$ is a non-degenerate
subspace of dimension $r-4\ge 1$.  Then, $V$ contains the
non-degenerate subspace $H_2\perp W$ of dimension $r-2\ge 3$.
So the restriction of $q$ to $V$ has rank at least $3$ and
Witt index at least $1$.  The conclusion follows by applying
Proposition \ref{equiv:prop:evitement}.
\end{proof}

\begin{lemma}
\label{proofv:lemma:recursion}
Suppose that a finite non-empty sequence $(\ux_1,\dots,\ux_k)$ of zeros of $q$
(in $\bQ^{n+1}$) satisfies the following conditions for $i=1,\dots,k$:
\begin{itemize}
 \item[(i)] $\ux_i\notin K$;
 \item[(ii)] $\Qspan{\ux_i,\ux_{i-1}}$ is totally
  isotropic if $i\ge 2$;
 \item[(iii)] $\ux_i \notin \Qspan{\ux_j,\dots,\ux_{i-1}}$
  where $j=\max\{1,i-n\}$ if $i\ge 2$;
 \item[(iv)] $\ux_i\notin \Qspan{\ux_{i-(n-1)},\dots,\ux_{i-1}}^\perp$
  if $i\ge n$.
\end{itemize}
Then there exists a zero $\ux_{k+1}\in\bQ^{n+1}$ of $q$ which satisfies the
same conditions for $i=k+1$.  Moreover, if $\ux_{k+1}$ is such a point, then
$a(\ux_{k+1}+b\ux_k)$ also satisfies these conditions for any $a\in\Qmult$
and any $b\in\bQ$ except possibly for one value of $b$.
\end{lemma}

In this statement, the properties that matter for us are (i),
(ii) which we motivated above, and the technical condition (iii)
which implies that any $n+1$ consecutive points among
$(\ux_1,\dots,\ux_k)$ are linearly independent over $\bQ$.  However,
it can be checked that, when $k\ge n$, the existence of a point
$\ux_{k+1}\in\bQ^{n+1}$ satisfying (i), (ii) and (iii) for
$i=k+1$ requires that (iv) hold for $i=k$. So, we need (iv) as
well for the recurrence step.

\begin{proof}
Set $V = \Qspan{\ux_k}^\perp$ and
$W_1 = \Qspan{\ux_j,\dots,\ux_{k}}$ where $j=\max\{1,k+1-n\}$.
Set also $W_2 = \Qspan{\ux_{k+2-n},\dots,\ux_{k}}$ if $k\ge n-1$.
We need to show that there exists a zero of $q$ in $V\setminus(K\cup W_1)$
if $k\le n-2$, in $V\setminus(K\cup W_1\cup W_2^\perp)$ if $k\ge n-1$.

We first note that $\dim_\bQ(V)=n$ since $\ux_k\notin K$, that
$\dim_\bQ(W_1)=k-j+1$ by Condition (iii), and that
$\dim_\bQ(W_2)=n-1$ if $k\ge n-1$ by the same condition.  We deduce that
$V\nsubseteq W_1$ because otherwise we would have $k\ge n$ and $V=W_1$
(by comparing dimensions), thus $\ux_k \in W_1^\perp \subseteq \langle
\ux_{k+1-n},\dots,\ux_{k-1}\rangle_\bQ^\perp$, which contradicts
Condition (iv) for $i=k$.  If $k\ge n-1$, we also have $V\nsubseteq W_2^\perp$
because otherwise we would get $W_2\subseteq V^\perp=K+\bQ\ux_k$ which is
impossible since $\dim_\bQ(K)\le n-3$.  For the same reason, we also have
$V\nsubseteq K$.  We conclude that $V \nsubseteq (K\cup W_1)$ if
$k\le n-2$, and $V \nsubseteq (K\cup W_1\cup W_2^\perp)$ if $k\ge n-1$.

By the above, Lemma \ref{proofv:lemma:evitement} yields a zero $\ux_{k+1}$
of $q$ with the requested properties if the rank of $q$ is at least $5$.
Suppose now that $q$ has rank $4$.  Then its Witt index is $m=2$.  Since
$\ux_k$ does not belong to $K$, it belongs to an hyperbolic plane
$H=\Qspan{\ux_k,\uy}$ for some zero $\uy$ of $q$ with $b(\ux_k,\uy)\neq 0$.
Write $\bQ^{n+1} = K \perp H \perp H'$ where $H'=\Qspan{\ux',\uy'}$
is another hyperbolic plane generated by zeros $\ux',\uy'$ of $q$ with
$b(\ux',\uy')\neq 0$.  This decomposition yields $V=K'\perp H'$ where
$K'= K + \bQ\ux_k$ is the kernel of the restriction $q|_V$ of $q$ to $V$.
So, $q|_V$ has rank $2$ and Witt index $1$.  Thus $V$ admits exactly two
maximal totally isotropic subspaces $U_1=K'+\bQ\ux'$ and $U_2=K'+\bQ\uy'$,
of dimension $n-1$, and the set of zeros of $q$ in $V$ is $U_1\cup U_2$.
So, we need to show that $U_1\cup U_2$ is not contained in $K\cup W_1$
if $k\le n-2$, and not contained in $K\cup W_1\cup W_2^\perp$ if
$k\ge n-1$.  For $k\le n-2$, this is clear since $K$ and $W_1$ have
dimensions strictly smaller than $n-1$.  Thus we may assume that
$k\ge n-1$.  Since $U_1+U_2=V$ is not contained in $W_1$ nor in
$W_2^\perp$, there is at most one index $r\in\{1,2\}$ such that
$U_r\subseteq W_1$ and at most one index $s\in\{1,2\}$ such that
$U_s\subseteq W_2^\perp$.  If such $r$ and $s$ exist and are distinct,
we may assume that $r=1$ and $s=2$ by permuting $\ux'$ and $\uy'$ is
necessary. Then, we have $U_1\subseteq W_1$ and $W_2\subseteq U_2^\perp=U_2$,
so $W_2$ and $U_2$ coincide since they have the same dimension $n-1$.
This is impossible because $W_1$ contains $W_2$ but not $U_2$.  Thus either
$r$ or $s$ does not exist or they are equal.  This means that, for some
$t\in\{1,2\}$, we have both $U_t\nsubseteq W_1$ and $U_t\nsubseteq W_2^\perp$,
so $U_t \nsubseteq K\cup W_1\cup W_2^\perp$ (since $\dim_\bQ(K)<n-1$), and we
are done.

Finally, suppose that a zero $\ux_{k+1}$ of $q$ satisfies Conditions
(i) to (iv) for $i=k+1$.  Then the point $a(\ux_{k+1}+b\ux_{k})$ is a zero
of $q$ which satisfies the same conditions for all $a\in\Qmult$ and all $b\in\bQ$
such that $\ux_{k+1}+b\ux_{k}\notin W_2^\perp$ if $k\ge n-1$.  Since
$\ux_{k+1}\notin W_2^\perp$, this excludes at most one value of $b$.
\end{proof}

We conclude with the proof of the following assertion from
Theorem \ref{results:thm2} (v).

\begin{proposition}
Let $\varphi\colon[1,\infty)\to(0,1]$ be a monotonically decreasing
function with
\[
 \lim_{X\to\infty} \varphi(X)=0
 \et
 \lim_{X\to\infty}X\varphi(X)=\infty.
\]
Then there exist
uncountably many points $\xi\in\Zli$ which satisfy $\cD_\xi(X;Z(\bQ))
\le \varphi(X)$ for all sufficiently large $X$.
\end{proposition}

\begin{proof}
Starting with a zero $\ux_1$ of $q$ in $\bZ^{n+1}\setminus K$,
Lemma \ref{proofv:lemma:recursion} allows us to construct recursively
a sequence $(\ux_i)_{i\ge 1}$ of zeros of $q$ in $\bZ^{n+1}$ which satisfy
Conditions (i) to (iv) of that lemma for each $i\ge 1$, such that,
upon setting $X_i=\norm{\ux_i}$ for each $i\ge 1$, we also have
 \begin{itemize}
 \item[(v)] $X_i>X_{i-1}$ when $i\ge 2$;
 \item[(vi)] $\disp \dist(\ux_{i},\ux_{i-1})
   \le \frac{1}{3} \min\big\{ 2X_{i-1}^{-1}\varphi(X_i),\,
          \dist(\ux_{i-1},\ux_{i-2}) \big\}$
   when $i\ge 3$.
\end{itemize}
Indeed, suppose that $\ux_1,\dots,\ux_{i-1}$ have been constructed for
some $i\ge 2$.  Then Lemma \ref{proofv:lemma:recursion} provides a zero
$\ux_i$ of $q$ satisfying Conditions (i) to (iv) of that lemma.  Upon
multiplying it by a suitable positive integer, we may assume that
$\ux_i\in\bZ^{n+1}$.  Let $\tux_i$ denote this particular zero.
By Lemma \ref{proofv:lemma:recursion}, the point $\ux_i=\tux_i+b\ux_{i-1}$
also satisfies these conditions for all but at most one value of $b$.
We find
\[
 \dist(\ux_{i},\ux_{i-1}) = \frac{\norm{\tux_i\wedge\ux_{i-1}}}{X_{i-1}X_i}
\]
with a numerator that is independent of the choice of $b$. As both
$X_i=\norm{\ux_i}$ and $X_i\varphi(X_i)$ go to infinity with $|b|$,
Conditions (v) and (vi) are fulfilled for $|b|$ large enough.

In $\bP^{n}(\bR)$, the image $([\ux_i])_{i\ge 1}$ of such a sequence
converges to a point $\xi$ with
\begin{equation}
 \label{proofv:prop:eq1}
 \dist(\xi,[\ux_{i-1}])
   \le \sum_{j=i}^\infty \dist(\ux_j,\ux_{j-1})
   \le \dist(\ux_i,\ux_{i-1}) \sum_{j=0}^\infty 3^{-j}
   = \frac{3}{2}\dist(\ux_i,\ux_{i-1})
\end{equation}
for each $i\ge 2$.  Since $[\ux_i]\in Z(\bQ)\subset Z$ for each $i\ge 1$
and since $Z$ is a closed subset of $\bP^n(\bR)$, the point $\xi$ belongs
to $Z$.  When $i\ge 3$, Condition (vi) combined with \eqref{proofv:prop:eq1}
yields
\begin{equation}
 \label{proofv:prop:eq2}
 D_\xi(\ux_{i-1})
  = X_{i-1}\dist(\xi,[\ux_{i-1}])
  \le \frac{3}{2}X_{i-1}\dist(\ux_i,\ux_{i-1})
  \le \varphi(X_i).
\end{equation}
In particular, we have $\lim_{i\to\infty} D_\xi(\ux_{i})=0$,
and so Lemma \ref{est:lemma:li} implies that $\xi\in\Zli$
because any $n+1$ consecutive points $\ux_i,\dots,\ux_{i+n}$ form
a basis of $\bQ^{n+1}$. Finally, for any $X\ge X_2$, there
exists an index $i\ge 3$ such that $X_{i-1}\le X < X_i$
and using \eqref{proofv:prop:eq2} we obtain
\[
  \cD_\xi(X;Z(\bQ))
   \le \cD_\xi(X_{i-1};Z(\bQ))
   \le D_\xi(\ux_{i-1})
   \le \varphi(X_i)\le \varphi(X).
\]

Thus, any sequence $(\ux_i)_{i\ge 1}$ as above yields a point
$\xi\in\Zli$ with the requested property.  To show that there
are uncountably many such points, consider any sequence
$\xi_1,\xi_2,\dots$ of these.  Then choose
$(\ux_i)_{i\ge 1}$ in $\bZ^{n+1}$ satisfying Conditions
(i) to (vi) as well as
\begin{equation}
 \label{proofv:prop:eq3}
 \dist(\ux_{i},\ux_{i-1})
   < \frac{2}{3} \min_{1\le j\le i} \dist(\xi_j,[\ux_{i-1}])
\end{equation}
for each $i\ge 2$.  This is possible since none of the points
$\xi_j$ is rational and therefore the right hand
side of the last inequality is non-zero. Let $\xi=\lim_{i\to\infty}[\ux_i]
\in\bP^n(\bR)$.  Combining \eqref{proofv:prop:eq1} and
\eqref{proofv:prop:eq3} for an arbitrary $i\ge 3$, we obtain
\[
 \dist(\xi,[\ux_{i-1}])
   < \min_{1\le j\le i} \dist(\xi_j,[\ux_{i-1}]),
\]
thus $\xi\notin\{\xi_1,\dots,\xi_i\}$.  So $\xi$ does not belong to
the sequence and therefore the set of points constructed in
this way is uncountable.
\end{proof}

%
%

\appendix

\section{An approximation lemma}
\label{sec:appA}

The following lemma is needed in Step 2 of Section \ref{sec:construction}.
It certainly occurs in the literature in various forms.  By lack of
an appropriate reference, we provide a short proof below.

\begin{lemme}
\label{lemme:approx}
Let $V$ be a real or complex inner product space of finite positive
dimension, let $T\colon V\to V$ be a linear operator on $V$, and
let $(\uv_i)_{i\ge 1}$ be a sequence in $V$.  Suppose that the
minimal polynomial of\/ $T$ admits a simple root $\alpha$ with
$|\alpha|>1$ and that any other root $\beta\in\bC$ of this
polynomial either has $|\beta|<1$ or is simple with $|\beta|=1$.
Suppose also that
\[
 \sum_{i=1}^\infty \|\uv_{i+1}-T(\uv_i)\| <\infty
\]
for the norm $\|\ \|$ of $V$.  Then there exists a vector $\uv\in V$
and a constant $C>0$ satisfying $T(\uv)=\alpha\uv$ and
$\|\uv_i-\alpha^i\uv\|\le C$ for each $i\ge 1$.
\end{lemme}

Recall that the minimal polynomial of $T$ is the monic polynomial
$m_T(x)$ of $\bC[x]$ of smallest degree such that $m_T(T)=0$.
It is a divisor of the characteristic polynomial of $T$ with the
same set of roots in $\bC$.

\begin{proof}
Upon extending $T$ by linearity to $\bC\otimes_\bR V$
in the case where $V$ is real, we reduce to the situation
where $V$ is complex.  As $V$ decomposes as a direct sum
of irreducible $T$-invariant subspaces, it suffices
to prove the lemma when $V$ itself is irreducible.
Then $T$ admits a single eigenvalue $\beta$ and its
minimal polynomial is $(x-\beta)^n$ for some $n>0$.
Put
\[
 \uu_1=\uv_1 \et \uu_{k+1}=\uv_{k+1}-T(\uv_k)
 \quad\text{for each $k\ge 1$.}
\]
By hypothesis, the sum $c_1=\sum_{k\ge 1}\|\uu_k\|$
is finite.  Moreover, we have
\begin{equation}
 \label{eq:U_k}
 \uv_i=\sum_{k=1}^i T^{i-k}(\uu_k) \quad (i\ge 1),
\end{equation}
If $|\beta|<1$, the operator norm $\|T^k\|$ tends
to $0$ as $k\to\infty$.  Thus there exists a constant
$c_2\ge 1$ such that $\|T^k\| \le c_2$ for all $k\ge 0$.
If $|\beta|=1$, we have $n=1$ thus $\|T^k\|=1$ for
each $k\ge 0$, and we may simply take $c_2=1$.  In both
cases, this yields
\[
 \|\uv_i\|\le c_2\sum_{k=1}^i \|\uu_k\|\le c_1c_2
 \quad (i\ge 1)
\]
and so the conclusion holds with $\uv=0$.  Finally if $|\beta|>1$,
we have $\beta=\alpha$ and $n=1$.  In this case the series
$\uv=\sum_{k=1}^\infty\alpha^{-k}\uu_k$ converges in $V$.
It satisfies $T(\uv)=\alpha\uv$ since $T=\alpha I$ and,
by \eqref{eq:U_k}, we find as requested
\[
 \|\uv_i-\alpha^i\uv\|
   =\Big\|\sum_{k=i+1}^\infty \alpha^{i-k}\uu_k\Big\|
   \le \sum_{k=i+1}^\infty\|\uu_k\|
   \le c_1
 \quad (i\ge 1).
 \qedhere
\]
\end{proof}


\end{document}